\documentclass{amsart}
\usepackage{amsmath}
\usepackage{amsfonts}
\usepackage{amssymb}
\usepackage{ascmac}
\usepackage{layout}
\usepackage{amsthm}

\usepackage{mathtools}

\title{Prikry-type forcings after collapsing a huge cardinal}

\author{Kenta Tsukuura}

\makeatletter
\@namedef{subjclassname@2020}{%
  \textup{2020} Mathematics Subject Classification}
\makeatother

\address{Doctoral Program in Mathematics, Degree Programs in Pure and Applied Sciences, Graduate School of Science and Technology, University of Tsukuba, Tsukuba, 305-8571, Japan}
\email{tukuura@math.tsukuba.ac.jp}

\thanks{This research was supported by Grant-in-Aid for JSPS Research Fellow Number 20J21103. The author is grateful to Masahiro Shioya for helpful discussions.}
							  
\newcommand{\force}{\Vdash}
\newcommand{\polar}[5]{\binom{#1}{#2}\to \binom{#3}{#4}_{#5}}

\newcommand{\npolar}[5]{\binom{#1}{#2}\not\to \binom{#3}{#4}_{#5}}

\theoremstyle{plain}
\newtheorem{thm}{Theorem}[section]

\newtheorem{rema}[thm]{Remark}
\newtheorem{clam}[thm]{Claim}
\newtheorem{lem}[thm]{Lemma}
\newtheorem{coro}[thm]{Corollary}
\newtheorem{prop}[thm]{Proposition}
\newtheorem{ques}[thm]{Question}

\begin{document}
\maketitle
\begin{abstract}
 Some models of combinatorial principles have been obtained by collapsing a huge cardinal in the case of the successors of regular cardinals. For example, saturated ideals, Chang's conjecture, polarized partition relations, and transfer principles for chromatic numbers of graphs. 

In this paper, we study these in the case of the successors of singular cardinals. In particular, we show that Prikry forcing preserves the centeredness of ideals but kills the layeredness. We also study $\polar{\mu^{++}}{\mu^{+}}{\kappa}{\mu^{+}}{\mu}$ and $\mathrm{Tr}_{\mathrm{Chr}}(\mu^{+++},\mu^{+})$ in the extension by Prikry forcing at $\mu$. 
\end{abstract}
\section{Introduction}
The existence of a saturated ideals over the successor cardinal is one of generic large cardinal axioms. Kunen~\cite{MR495118} constructed a model of a saturated ideal over $\aleph_1$ by collapsing a huge cardinal. Let $\mu$ be a regular cardinal. Kunen's construction has been modified to obtain models in which
\begin{itemize}
 \item (Laver~\cite{MR673792}) $\mu^{+}$ carries a strongly saturated ideal.
 \item (Foreman--Laver~\cite{MR925267}) $\mu^{+}$ carries a centered ideal. 
 \item (Foreman--Magidor--Shelah~\cite{MR942519}) $\mu^{+}$ carries a layered ideal.
\end{itemize}
Strong saturation, centeredness, and layerdness are strengthenings of saturation. See Section 2 for the definitions. The following hold in respective models.
\begin{itemize}
 \item (Laver~\cite{MR673792}) $(\mu^{++},\mu^{+}) \twoheadrightarrow (\mu^{+},\mu)$. 
 \item (Laver~\cite{MR673792}) $\polar{\mu^{++}}{\mu^{+}}{\mu^{+}}{\mu^{+}}{\mu}$. 
 \item (Foreman--Laver~\cite{MR925267}) Every graph of size and chromatic number $\mu^{++}$ has a subgraph of size and chromatic number $\mu^{+}$.
\end{itemize}

In this paper, we will consider these principles for singular $\mu$. Note that the above models are obtained by $\mu$-directed closed posets. Then, by Laver's theorem~\cite{Laver}, we can get models in which $\mu$ is supercompact as well. This enables us to use Prikry-type forcings. The problem is whether the above principles are preserved by Prikry-type forcings. Foreman~\cite{MR730584} showed that a class of posets including some Prikry-type forcings preserve the existence of saturated ideals.
\begin{lem}[Foreman~\cite{MR730584}]\label{foremanfst}
 Suppose that $\mu$ is a regular cardinal and $I$ is a saturated ideal over $\mu^+$. Then every $\mu$-centered poset forces the ideal $\overline{I}$ generated by $I$ is a saturated ideal. 
\end{lem}
Foreman~\cite{MR2583810} also constructed a model in which every successor cardinal carries a centered ideal. In his proof, he claimed an analogue of Lemma \ref{foremanfst} for centered ideals without proof. In this paper, we give a proof of this in Lemma \ref{termcentered}. We also study the layeredness of $\overline{I}$ in the extension by some Prikry-type forcings. We will show that $\overline{I}$ is always \emph{not} layered. We have 
\begin{thm}\label{maintheorem1}
 Suppose that $\mu < \kappa$ is a measurable cardinal and $\mu^{+}$ carries a saturated ideal $I$, and $2^{\mu} = \mu^{+}$. Then Prikry forcing, Woodin's modification~\cite{MR1007865}, and Magidor forcing at $\mu$ force that 
\begin{enumerate}
 \item (Foreman~\cite{MR2583810}) $\overline{I}$ is centered if $I$ is centered in $V$.
 \item $\overline{I}$ is $\emph{not}$ layered.
\end{enumerate}
\end{thm}

We also study the preservation of polarized partition relations by Prikry forcing. 
\begin{thm}\label{maintheorem2}
 Prikry forcing preserves the following:
\begin{enumerate}
 \item $\polar{\mu^{++}}{\mu^{+}}{n}{\mu^{+}}{\mu}$ for each $n < \omega$.
 \item $\polar{\mu^{++}}{\mu^{+}}{\nu}{\mu^{+}}{\mu}$ for each regular $\nu < \mu$.
 \item $\npolar{\mu^{++}}{\mu^{+}}{n}{\mu^{+}}{\mu}$ for each $n < \omega$.
 \item $\npolar{\mu^{++}}{\mu^{+}}{\nu}{\mu^{+}}{\mu}$ for each regular $\nu < \mu$.
 \item $\npolar{\mu^{++}}{\mu^{+}}{\mu^{+}}{\mu^{+}}{\mu}$.

\end{enumerate}
\end{thm}

Our proof of Theorem \ref{maintheorem2} (1) shows that Cohen forcing $\mathrm{Add}(\aleph_0,\aleph_1)$ preserves $\polar{\aleph_2}{\aleph_1}{2}{\aleph_1}{\aleph_0}$ as well. Note that $\mathrm{Add}(\aleph_0,\aleph_1)$ forces $\npolar{\aleph_2}{\aleph_1}{\aleph_0}{\aleph_1}{2}$ as shown by Hajnal--Juhasz (see Theorem \ref{hajnaljuhasz}). This enables us to answer a question by Garti~\cite[Question 1.11]{MR4101445}. Combining these result, we show

\begin{thm}\label{maintheorem3}
 Suppose that $\kappa$ is a huge cardinal and $\mu < \kappa$ is a supercompact cardinal. Then there is a poset which forces that $\kappa = \aleph_{\omega+2}$, $\mu = \aleph_{\omega}$,
\begin{enumerate}
 \item $\aleph_{\omega+1}$ carries an ideal $I$ that is centered but \emph{not} layered, and
 \item $I$ is $(\aleph_{\omega+2},\aleph_{n},\aleph_{n})$-saturated for all $n < \omega$.
 \item $\polar{\aleph_{\omega+2}}{\aleph_{\omega+1}}{\aleph_{n}}{\aleph_{\omega+1}}{\aleph_{\omega}}$ for all $n < \omega$, 
 \item $\npolar{\aleph_{\omega+2}}{\aleph_{\omega+1}}{\aleph_{\omega+1}}{\aleph_{\omega+1}}{\aleph_{\omega}}$, and
 \item $(\aleph_{\omega+2},\aleph_{\omega+1}) \twoheadrightarrow (\aleph_{\omega+1},\aleph_{\omega})$.
\end{enumerate}
\end{thm}

We will show that the existence of a $\mu^{++}$-centered ideal over $[\mu^{+++}]^{\mu^{+}}$ implies that every graph of size and chromatic number $\mu^{+++}$ has a subgraph of size and chromatic number $\mu^{+}$. We also give a model in which $[\aleph_{\omega+3}]^{\aleph_{\omega+1}}$ carries a $\aleph_{\omega+2}$-centered ideal by generalizing Theorem \ref{maintheorem1} (1) (see Lemma \ref{generalizedpreservation}). We have
\begin{thm}\label{maintheorem4}
 Suppose that $\kappa$ is a huge cardinal and $\mu < \kappa$ is a supercompact cardinal. Then there is a poset which forces that $\kappa = \aleph_{\omega+2}$, $\mu = \aleph_{\omega}$,
\begin{enumerate}

 \item $[\aleph_{\omega+3}]^{\aleph_{\omega+1}}$ carries a normal, fine, $\aleph_{\omega+1}$-complete $\aleph_{\omega+2}$-centered ideal, and
 \item Every graph of size and chromatic number $\aleph_{\omega+3}$ has a subgraph of size and chromatic number $\aleph_{\omega+1}$. 
\end{enumerate}
\end{thm}

The structure of this paper is as follows:
In Section 2, we recall basic facts of some saturation properties, duality theorem, and Prikry-type forcings. The duality theorem plays a central role when we study the saturation property of ideals in some extension. In Section 3, we show Theorem \ref{maintheorem1}. In Section 4, we will see that Chang's conjecture and the existence of saturated ideals imply some polarized partition relations, respectively. The proof of Theorem \ref{maintheorem2} is contained in Section 4. Combining the results in Section 3 and Section 4, we give a proof of Theorem \ref{maintheorem3} in Section 5. In Section 6, we introduce the transfer principle $\mathrm{Tr}_{\mathrm{Chr}}(\lambda,\kappa)$ for the chromatic number of graphs. We generalize Theorem \ref{maintheorem1}.(1) to an ideal over $Z \subseteq \mathcal{P}(\lambda)$ and show that the existence of a $\mu^{++}$-centered ideal over $[\mu^{+++}]^{\mu^{+}}$ implies $\mathrm{Tr}_{\mathrm{Chr}}(\mu^{+++},\mu^{+})$. By using these facts, we give a proof of Theorem \ref{maintheorem4}.

\section{Preliminaries}
In this section, we recall basic facts of the saturation properties of ideals and Prikry-type forcings, and some combinatorial principles. 
We use~\cite{MR1994835} as a reference for set theory in general. For more on the topic of saturated ideal and Prikry-type forcing, we refer to~\cite{MR2768692} and~\cite{Gitik}, respectively. 

Our notation is standard. We use $\kappa,\lambda$ to denote a regular cardinal unless otherwise stated. We also use $\mu,\nu$ to denote a cardinal, possibly finite, unless otherwise stated. For $\kappa < \lambda$, $E^{\lambda}_\kappa$, $E^{\lambda}_{>\kappa}$ and $E^{\lambda}_{\leq\kappa}$ denote the set of all ordinals below $\lambda$ of cofinality $\kappa$, $>\kappa$ and $\leq\kappa$, respectively. We also write $[\kappa,\lambda) = \{\xi \mid \kappa \leq \xi < \lambda\}$. By $\mathrm{Reg}$, we mean the class of regular cardinals. 

For every poset $P$, We identify $P$ with its separative quotient. That is, $p\leq q \leftrightarrow \forall r \leq p(r \parallel q) \leftrightarrow p \force q \in \dot{G}$ for all $p,q \in P$. Here, $\dot{G}$ is the canonical name for a generic filter. We say that $P$ is well-met if $\prod Z \in P$ for all $Z \subseteq P$ that has a lower bound in $P$. Note that all poset we will deal with in this paper is well-met.

For a complete embedding $e:P \to Q$, the quotient forcing is defined by $P \force Q/ e ``\dot{G} = \{q \in Q \mid \forall r \in \dot{G}(e(r) \parallel q)\}$ ordered by $\leq_{Q}$. $P \ast Q / e ``\dot{G}$ is forcing equivalent with $Q$. We also write $Q / \dot{G}$ for $Q / e ``\dot{G}$ if $e$ is clear from the context. If the inclusion mapping from $P$ to $Q$ is a complete embedding, we say that $P$ is a complete suborder of $Q$, denoted by $P \lessdot Q$. The completion of $P$ is a complete Boolean algebra $\mathcal{B}(P)$ such that $P \lessdot \mathcal{B}(P)$ and $P$ is a dense subset of $\mathcal{B}(P)$. $\mathcal{B}(P)$ is unique up to isomorphism.

For a given complete embedding $e:P \to Q$, the mapping $q \mapsto \prod\{p \in P \mid \forall r \leq p(e(r)\parallel q)\}$ define a projection $\pi:Q \to \mathcal{B}(P)$. It is easy to see that $\pi(e(p)) = p$ and $e(\pi(q)) \geq q$. 

We often use the following theorem.
\begin{thm}[Laver~\cite{Laver}]\label{laverind}
 If $\mu$ is supercompact then there is a poset $P$ such that 
\begin{enumerate}
 \item $P \subseteq V_{\mu}$, 
 \item $P \force \mu$ is supercompact.
 \item For every $P$-name $\dot{Q}$ with $P \force \dot{Q}$ is $\mu$-directed closed, $P \ast \dot{Q} \force \mu$ is supercompact.
\end{enumerate}
\end{thm}
We say that a supercompact cardinal $\mu$ is indestructible if, for every $\mu$-directed closed poset $Q$, $Q \force \mu$ is supercompact. If $\mu$ is supercompact and $\kappa > \mu$ is huge then we can force $\mu$ to be indestructible without destroying the hugeness of $\kappa$.

\subsection{Saturation of Ideals}

For cardinals $\mu \leq \kappa \leq \lambda$, we say that $P$ has the $(\lambda,\kappa,<\mu)$-c.c. if, for every $X \in [P]^{\lambda}$ there is a $Y \in [X]^{\kappa}$ such that $Z$ has a lower bound for all $Z \in [Y]^{<\mu}$. By the $(\lambda,\kappa,\mu)$-c.c., we mean the $(\lambda,\kappa,<\mu^{+})$-c.c. $\lambda$-c.c. and $\lambda$-Knaster are the same as the $(\lambda,2,2)$-c.c. and the $(\lambda,\lambda,2)$-c.c., respectively. 

$P$ is $(\lambda,<\nu)$-centered if and only if $P = \bigcup_{\alpha < \lambda} P_{\alpha}$ for some $<\nu$-centered subsets $P_{\alpha} \subseteq P$. $<\nu$-centered subset is a $C \subseteq P$ such that $\forall Z \in [C]^{<\nu}(Z$ has a lower bound in $P)$. We call such a family of centered subsets a centering family of $P$. This is equivalent with the existence of a function $f:P \to \lambda$ such that $f^{-1}\{\alpha\}$ is a centered subset for each $\alpha$. We call such $f$ a centering function of $P$. By $\lambda$-centered, we means $(\lambda,<\omega)$-centered. Note that if $P$ is $\lambda$-centered then $|P| \leq 2^{\lambda}$. Indeed, for a centering family $\{P_{\alpha}\mid \alpha < \lambda\}$, $p \mapsto \{\alpha \mid p \in P_{\alpha}\}$ is an injective mapping from $P$ to $\mathcal{P}(\lambda)$.

If $P$ is well-met then the $(\lambda,<\nu)$-centeredness of $P$ is equivalent with $P$ can be covered by $\lambda$-many $<\nu$-complete filters. 

For a stationary subset $S \subseteq \lambda$, let us introduce the $S$-layeredness, which was originally introduced by Shelah~\cite{MR942519}. We say that $P$ is $S$-layered if there is a club $C \subseteq [\mathcal{H}_{\theta}]^{<\lambda}$ such that, for all $M \in C$, if $\sup (M \cap \lambda) \in S \to M \cap P \mathrel{\lessdot} P$ for any sufficiently large regular $\theta$. We will consider the $S$-layeredness of complete Boolean algebra $P$. Note that $M \cap P$ is a Boolean subalgebra of $P$ but is not necessarily a complete Boolean subalgebra of $P$ even if $M \cap P \mathrel{\lessdot} P$.

\begin{lem}\label{charlayered}
 For a stationary subset $S \subseteq \lambda$ and poset $P$ of size $\leq \lambda$, the following are equivalent:
\begin{enumerate}
 \item $P$ is $S$-layered.
 \item There is an $\subseteq$-increasing sequence $\langle P_\alpha \mid \alpha < \lambda \rangle$ with the following properties:
       \begin{enumerate}
	\item $P = \bigcup_{\alpha < \lambda}P_{\alpha}$.
	\item $P_{\alpha} \lessdot P$ and $|P_\alpha| < \lambda$ for all $\alpha < \lambda$. 
	\item There is a club $C \subseteq \lambda$ such that $\forall \alpha \in S \cap C (P_\alpha = \bigcup_{\beta < \alpha}P_\alpha)$. 
       \end{enumerate}
 \item There is an $\subseteq$-increasing continuous sequence $\langle P_\alpha \mid \alpha < \lambda \rangle$ with the following properties:
       \begin{enumerate}
	\item $P = \bigcup_{\alpha < \lambda}P_{\alpha}$.
	\item $P_{\alpha} \subseteq P$ and $|P_\alpha| < \lambda$ for all $\alpha < \lambda$. 
	\item There is a club $C \subseteq \lambda$ such that $\forall \alpha \in S \cap C (P_\alpha \lessdot P)$. 
       \end{enumerate}
\end{enumerate}
\end{lem}
\begin{proof}
 For the equivalence between (2) and (3), we refer to \cite{preprint}. It is easy to see that (1) and (3) are equivalent.
\end{proof}

The original definition of $S$-layeredness of $P$ by Shelah is (3). If we define the $S$-layeredness by (3) then $\mathcal{B}(P)$ is not necessarily $S$-layered even if $P$ is. By our definition, the $S$-layeredness of $P$ is equivalent with that of $\mathcal{B}(P)$.

\begin{lem}
 Suppose that there is a complete embedding $\tau:P \to Q$. 
\begin{enumerate}
 \item If $Q$ has the $(\lambda,\lambda,<\nu)$-c.c. then so does $P$. 
 \item If $Q$ is $S$-layered for some stationary $S \subseteq \lambda$, then so is $P$.
 \item If $Q$ is $(\lambda,<\nu)$-centered, then so is $P$. 
\end{enumerate}
\end{lem}
\begin{proof}
 We may assume that $P$ and $Q$ are Boolean algebra (not necessarily complete). We show only (2). It suffices to show that $Q \cap M \lessdot Q$ implies $P \cap M \lessdot P$ for club many $M \in [\mathcal{H}_{\theta}]^{<\lambda}$. We fix $M \prec \mathcal{H}_{\theta}$ with $P,Q,\tau \in M$. Suppose $Q \cap M \lessdot Q$. Let $p \in P$ be arbitrary. $\tau(p)$ has a reduct $q$ in $Q \cap M$. By the elementarity of $M$, we can choose a reduct $p_0 \in P \cap M$ of $q$ (in the sense of $\tau:P \to Q$). For every $r \in P \cap M$ with $r \leq p_0$, $\tau(r) \cdot q \not= 0$. Since $q$ is a reduct of $\tau(p)$, $\tau(r) \cdot q \cdot \tau(p) \not=0$, which in turn implies $r \cdot p\not= 0$ in $P$. Therefore $p_0 \in P \cap M$ is a reduct of $p \in P$. 
\end{proof}

In this paper, by an ideal over $\mu^{+}$, we mean a normal, fine, and $\mu^{+}$-complete ideal. For an ideal $I$ over $Z \subseteq \mathcal{P}(\lambda)$, we say $I$ is fine whenever $\{x \in Z\mid \alpha \not\in x\} \in I$ for all $\alpha \in X$. $\mathcal{P}(Z) / I$ denotes the poset with the underling set $I^{+} = \mathcal{P}(Z) \setminus I$ and the ordered on $I^{+}$ is defined by $A \leq B\leftrightarrow A \setminus B \in I$. 

For an ideal $I$ over $\mu^{+}$, $I$ is saturated if $\mathcal{P}(\mu^{+}) / I$ has the $\mu^{++}$-c.c. $I$ is $(\kappa,\lambda,<\nu)$-saturated if $\mathcal{P}(\mu^{+}) / I$ has the $(\kappa,\lambda,<\nu)$-c.c. We also say $I$ is strongly saturated if $I$ is $(\mu^{++},\mu^{++},\mu)$-saturated. 

$I$ is centered and layered if $\mathcal{P}(\mu^{+}) / I$ is $\mu$-centered and $S$-layered for some $S \subseteq E_{\mu^{+}}^{\mu^{++}}$, respectively. For other saturation property $\psi$ of posets, we use a phrase that an ideal (over $Z$) is $\psi$ as well. 

In Sections 3 and 6, we will use Theorem \ref{duality}. For a $\mu^{+}$-c.c. poset $P$ and a normal and $\mu^{+}$-complete ideal $I$ over $Z$, we can consider a $P$-name $\overline{I}$ for the ideal generated by $I$. That is, $P \force \overline{I} = \{A \subseteq Z \mid \exists B \in I(A \subseteq B)\}$. $\overline{I}$ is normal and $\mu^{+}$-complete in the extension. We are interested in the extent of saturation of $\overline{I}$. Theorem \ref{duality}, which is one of special cases of the duality theorem, is useful to study $\dot{\mathcal{P}}(Z) /\overline{I}$. For details of the duality theorem, we refer to~\cite{MR3279214},~\cite{MR2768692}, or~\cite{MR3038554}. Here, we give a direct proof for understanding proofs in Section \ref{centeredlayered}. 
\begin{thm}[Foreman~\cite{MR3038554}]\label{duality}
 For a normal, fine, $\mu^{+}$-complete $\lambda^{+}$-saturated ideal over $Z \subseteq \mathcal{P}(\lambda)$ (for some $\lambda > \mu$) and $\mu^{+}$-c.c. $P$, there is a dense embedding $\tau$ such that:
 \[
  \begin{array}{rccc}
  \tau:& P \ast \dot{\mathcal{P}}(Z) / \overline{I} &\longrightarrow & \mathcal{B}(\mathcal{P}(Z) / I \ast \dot{j}(P)) \\
  & \rotatebox{90}{$\in$} & &\rotatebox{90}{$\in$} \\
  & \langle p,\dot{A} \rangle & \longmapsto & e(p)\cdot||[\mathrm{id}] \in \dot{j}(\dot{A})||
 \end{array}
 \]
 Here, $e(p) = \langle 1,\dot{j}({p}) \rangle$ is a complete embedding from $P$ to $\mathcal{P}(Z) / I \ast \dot{j}(P)$ and $\dot{j}:V \to \dot{M}$ denotes the generic ultrapower mapping by $\mathcal{P}(Z) / I$. In particular, $P \force\dot{\mathcal{P}}(Z) / \overline{I} \simeq \mathcal{B}(\mathcal{P}(Z) / I \ast \dot{j}(P) / e ``\dot{H}_0)$. Here, $\dot{H}_0$ is the canonical $P$-name for a generic filter.
\end{thm}

\begin{proof}
 We may assume that $P$ is a complete Boolean algebra.  Note that it follows that $e$ is complete since $P$ has the $\mu^{+}$-c.c. and $\mathrm{crit}(\dot{j}) = \mu^{+}$. Indeed, for every maximal anti-chain $\mathcal{A} \subseteq P$, by $|\mathcal{A}| <\mu^{+}$, 
 \begin{align*}
  \textstyle\sum e ``\mathcal{A} &= \textstyle\sum_{p \in \mathcal{A}}||\langle 1,\dot{j}(p) \rangle \in \dot{H}|| = \textstyle\sum_{p \in \mathcal{A}}||\langle 1,\dot{j}(p) \rangle \in \dot{H}||\\ & = ||j `` \mathcal{A} \cap \dot{H} \not= \emptyset|| = ||j(\mathcal{A}) \cap \dot{H}\not= \emptyset|| \\ &= 1
 \end{align*} 
Here, $\dot{G} \ast \dot{H}$ is the canonical $\mathcal{P}(Z) / I \ast \dot{j}(P)$-name for a generic filter.

 Our proof consists of two parts. First, we will give a $P$-name $\dot{J}$ and a dense embedding $\tau_0:P \ast \dot{\mathcal{P}}(Z) / \dot{J} \to \mathcal{B}(\mathcal{P}(Z) / I \ast \dot{j}(P))$. After that, we will see $P \force \dot{J} = \overline{I}$ and $\tau_0 = \tau$. 

Let $\dot{J}$ be a $P$-name defined by $P \force \dot{J} \subseteq \dot{\mathcal{P}}(Z)$ and 
\begin{center}
 $A \in \dot{J}$ if and only if $\mathcal{P(\mu^{+})} / I \ast \dot{j}(P) / e ``\dot{H}_0 \force [\mathrm{id}] \not\in \dot{j}(\dot{A})$. 
\end{center}
It is easy to see that $\dot{J}$ is forced to be an ideal. 
Define $\tau_0:P \ast \dot{\mathcal{P}}(Z) / \dot{J} \to \mathcal{B}(\mathcal{P}(Z) / I \ast \dot{j}(P))$ by $\tau_0(p,\dot{A}) = e(p) \cdot ||[\mathrm{id}] \in \dot{j}(\dot{A})||$. By the definition of $\dot{J}$, if $P \not\force \dot{A} \in \dot{J}$ then by some element then $||[\mathrm{id}] \in \dot{j}(\dot{A})|| \not= 0$.

 Let us see the range of $\tau_0$ is a dense subset. Let $\langle B,\dot{q} \rangle \in \mathcal{P}(Z) / I \ast \dot{j}(P)$ be an arbitrary element. Since $I$ is $\lambda^{+}$-saturated, we can choose $f:Z \to P$ such that $B \force \dot{q} = [f]$. Since $e$ is complete, $\langle B,\dot{q}\rangle$ has a reduct $p \in P$. For every $r \leq p$, $e(r) \cdot \langle B,\dot{q} \rangle \not= 0$ and this forces $\dot{j}(f)([\mathrm{id}]) = [f] = \dot{q} \in \dot{H} = \dot{j}(\dot{H}_0)$. Therefore $p$ forces that $\mathcal{P(\mu^{+})} / I \ast \dot{j}(P) / e ``\dot{H}_0 \not\force [\mathrm{id}] \not\in \dot{j}(\{x \in B \mid f(x) \in \dot{H}_0\})$''. Thus, there is a $P$-name $\dot{A}$ such that $P \force \dot{A} \in \dot{J}^{+}$ and $p \force \dot{A} = \{x \in B \mid f(x) \in \dot{H}_0\}$. It is easy to see $\tau(p,\dot{A}) = e(p) \cdot ||[\mathrm{id}] \in \dot{j}(\dot{A})|| \leq \langle B,\dot{q} \rangle$, as desired.

 Lastly, we claim that $P \force \dot{J} = \overline{I}$. $P \force \overline{I} \subseteq \dot{J}$ is clear. To show $P \force \dot{J} \subseteq \overline{I}$, let us consider $p \force \dot{C} \in \overline{I}^{+}$. We let $D = \{x \in Z \mid ||x \in \dot{C}||_{P} \cdot p \not= 0\} \in I^{+}$. $D$ forces $\dot{j}(p) \cdot ||[\mathrm{id}] \in \dot{j}(\dot{C})||_{\dot{j}(P)}^{\dot{M}} \not= 0$ . Let $\dot{q}$ be a $\mathcal{P}(Z)/ I$-name such that $\force \dot{q} \in \dot{j}(P)$ and $D\force \dot{q} = \dot{j}(p) \cdot ||[\mathrm{id}] \in \dot{j}(\dot{C})||^{\dot{M}}_{\dot{j}(P)}$. Let $r$ be a reduct of $\langle D,\dot{q}\rangle \in \mathcal{P}(Z) / I \ast \dot{j}(P)$. It is easy to see that $r\leq p$ and $r \force$ ``$\langle D,\dot{q}\rangle \leq ||[\mathrm{id}] \in \dot{j}(\dot{C})||$ in the quotient forcing''. By the definition of $\dot{J}$, $r \force \dot{C} \in \dot{J}^{+}$, as desired. Of course, $\tau = \tau_0$. The proof is completed.
\end{proof}
The following corollaries have nothing to do with proofs of main theorems. But we introduce these.
\begin{coro}[Baumgartner--Taylor~\cite{MR654852}]\label{baumgartnertaylor}
 For a saturated ideal $I$ over $\mu^{+}$ and $\mu^{+}$-c.c. $P$, the following are equivalent:
\begin{enumerate}
 \item $P \force \overline{I}$ is saturated.
 \item $\mathcal{P}(\mu^{+})/ I \force \dot{j}(P)$ has the $(\mu^{++})^{V}$-c.c.
\end{enumerate}
\end{coro}
In particular, for a saturated ideal $I$ over $\mu^{+}$, if $P$ is $\mu$-centered then $P \force \overline{I}$ is saturated. Corollary \ref{baumgartnertaylor} is one of the improvements of Lemma \ref{foremanfst}. Some of Prikry-type forcings are $\mu$-centered. Therefore $\overline{I}$ is forced to be saturated by these posets. Using Theorem \ref{duality}, we can get necessary conditions of $\overline{I}$ to become centered or strongly saturated. We prove the following corollary because this is the motivation for Theorem \ref{maintheorem1}.
\begin{coro}
 For a saturated ideal $I$ over $\mu^{+}$ and $\mu$-centered poset $P$, if $(\mu^{+})^{\mu} = \mu^{+}$ then the following holds.
\begin{enumerate}
 \item If $P \force \overline{I}$ is centered then $I$ is centered.
 \item If $P \force \overline{I}$ is strongly saturated then $I$ is strongly saturated.
\end{enumerate}
\end{coro}
\begin{proof}
We may assume that $P$ is a Boolean algebra (not necessarily complete).

Let $e:P \to (\mathcal{P}(\mu^{+})/ I \ast \dot{j}(P))$ be a complete embedding given in Theorem \ref{duality}. Then $P \force \mathcal{P}(\mu^{+}) / \overline{I} \simeq \mathcal{P}(\mu^{+})/ I \ast \dot{j}(P) / \dot{G}$. For every $A \in I^{+}$, $P \force \langle A,\dot{1}\rangle \in \mathcal{P}(\mu^{+})/ I \ast \dot{j}(P) / \dot{G}$. Indeed, for every $p \in P$, $e(p) \cdot \langle A,1\rangle = \langle A,\dot{j}(p) \rangle \in \mathcal{P}(\mu^{+}) / I \ast \dot{j}(P)$. It is easy to see that $\mathcal{P}(\mu^{+}) / I$ is completely embedded in $\mathcal{P}(\mu^{+}) / I \ast \dot{j}(P) / \dot{G}$ by a mapping $A \mapsto \langle A,\dot{1}\rangle$.

We check (1). Let $\langle P_\alpha \mid \alpha < \mu\rangle$ be a centering family of $P$ and let $\dot{f}$ be a $P$-name for a centering function of $(\mathcal{P}(\mu^{+}) / I)^{V}$. We may assume that each $P_{\alpha}$ is a filter. For each $A \in I^{+}$, define $f(A) = \langle \xi \mid \alpha < \lambda, \exists q \in P_{\alpha}(q \force \dot{f}(A) = \xi) \rangle$. By $(\mu^{+})^{\mu} =\mu^{+}$, we identify the range of $f$ with $\mu^{+}$. It is easy to see that $f$ works as a centering function in $V$. Therefore $I$ is centered. 

Let us see (2). Similarly, $P$ forces $(\mathcal{P}(\mu^{+}) / I)^{V}$ has the $(\mu^{++},\mu^{++},\mu)$-c.c. For every $X \in [I^{+}]^{\mu^{++}}$, there is a $P$-name $\dot{Y}$ such that $P \force \dot{Y} \in [X]^{\mu^{++}}$ and $\forall Z \in [\dot{Y}]^{\mu}(\bigcap Z \in I^{+})$. Since $P$ is $\mu$-centered, $|P| \leq 2^{\mu} = \mu^{+}$. Therefore there is an $Y \in [X]^{\mu^{++}}$ and $p \in P$ such that $p \force Y \in [\dot{Y}]^{\mu^{++}}$. $Y$ works as a witness. 
\end{proof}

\subsection{Prikry-type forcings}\label{prikrytypeforcings}
Modifications of Prikry forcing are called Prikry-type forcings. Original Prikry forcing was introduced by Prikry~\cite{prikry}. For a given filter $F$ over $\mu$, $\mathcal{P}_{F}$ is $[\mu]^{<\omega} \times F$ ordered by $\langle a, X \rangle \leq \langle b, Y\rangle$ if and only if $a \supseteq b$, $a \cap (\max{b} + 1) = b$ and $a\setminus b \cup X \subseteq Y$. It is easy to see that $\mathcal{P}_{U}$ is $(\mu,<\mu)$-centered. Prikry forcing is $\mathcal{P}_{U}$ for some normal ultrafilter $U$. Prikry forcing preserves all cardinals and forces $\mathrm{cf}(\mu) = \omega$. 
\begin{lem}\label{prikrylem}
 Suppose that $U$ is a normal ultrafilter over $\mu$. For every $a \in [\mu]^{<\mu}$ and statement $\sigma$ in the forcing language of $\mathcal{P}_{U}$, there is $Z \in U$ such that $\langle a,Z \rangle$ decides $\sigma$. That is, $\langle a,Z \rangle \force \sigma$ or $\langle a,Z\rangle \force \lnot \sigma$. 
\end{lem}
\begin{proof}
 See~\cite{Gitik} or~\cite{MR4404936}. 
\end{proof}

We often use the following variation of Lemma \ref{prikrylem}. 
\begin{lem}\label{prikrycondi}
 Suppose that $U$ is a normal ultrafilter over $\mu$ and $\mathcal{A} \subseteq \mathcal{P}_{{U}}$ is a maximal anti-chain below $\langle a,X \rangle$. Then there are $n$ and $X\supseteq Z\in U$ such that $\{\langle b,Y \rangle \in \mathcal{A}\mid |b| = n\}$ is a maximal anti-chain below $\langle a,Z \rangle$. 
\end{lem}
\begin{proof}
 Suppose that $\mathcal{A}$ is a maximal anti-chain below $\langle a,X \rangle$. For each $n < \omega$, by Lemma \ref{prikrylem}, there is a $Z_{n} \in U$ such that $\langle a,Z_n \rangle$ decides $\exists \langle b,Y\rangle \in \dot{G} \cap \mathcal{A}(|b| = n)$. $Z = X \cap\bigcap_{n}Z_n$ works.
\end{proof}

The following lemma will be used in a proof of Theorem \ref{maintheorem1}(2) and Proposition \ref{hugeprop}.
\begin{lem}\label{prikryequiv}
 For a posets $P \lessdot Q$, let $\dot{U}$ and $\dot{W}$ be a $P$-name and a $Q$-name for a filter over $\mu$, respectively. If $Q \force \dot{U} \subseteq \dot{W}$ and $\dot{W}$ is a normal ultrafilter over $\mu$, then the following are equivalent:
\begin{enumerate}
 \item $P \ast \mathcal{P}_{\dot{U}} \lessdot Q \ast \mathcal{P}_{\dot{W}}$.
 \item $P \force \dot{U}$ is ultrafilter.
\end{enumerate}
\end{lem}
\begin{proof}
 We may assume that $P$ and $Q$ are Boolean algebras.

Let us show the forward direction. We show contraposition. Suppose that there are $p \in P$ and $\dot{X}$ such that $p \force \dot{X} \not\in \dot{U}$ and $\mu \setminus \dot{X} \not\in \dot{U}$. Then there is a extension $q \in Q$ of $p$ which decides $\dot{X} \in \dot{W}$. We may assume that $q$ forces $\dot{X} \in \dot{W}$. We claim that there is no reduct of $\langle q,\langle \emptyset,\dot{X} \rangle \rangle$ in $P \ast \mathcal{P}_{\dot{U}}$. 

For any $\langle r,\langle a,\dot{Y} \rangle \rangle \in P \ast \mathcal{P}_{\dot{U}}$, if $r$ is not a reduct of $q$ (in the sense of $P \lessdot Q$), there is nothing to do. Suppose that $r$ is a reduct of $q$. Then we have $r \leq p$. By $r \force \dot{X} \not\in \dot{U}$, $r \force |\dot{Y} \setminus \dot{X}| = \mu$. Choose $r' \leq r$ and $\alpha$ with $r \force \alpha \in \dot{Y} \setminus \dot{X} \cup (\max a + 1)$. Then $\langle r',\langle a \cup \{\alpha\},\dot{Y} \rangle \rangle \leq \langle r,\langle a,\dot{Y} \rangle \rangle$ does not meet with $\langle q,\langle \emptyset,\dot{X} \rangle \rangle$, as desired.

 The inverse direction follows by Lemma \ref{prikrycondi}. For a maximal anti-chain $\mathcal{A} \subseteq P \ast \mathcal{P}_{\dot{U}}$, consider $P$-name $\dot{\mathcal{B}}$ such that $P \force \dot{\mathcal{B}} = \{\langle a,X \rangle \mid \exists p \in \dot{G}(\langle p,\langle a,X\rangle \rangle \in \mathcal{A})\}$. $\dot{\mathcal{B}}$ is forced to be a maximal anti-chain. It is enough to prove that $Q \force \dot{\mathcal{B}}$ is maximal anti-chain below $\mathcal{P}_{\dot{W}}$. For every $p \force \langle a,\dot{X} \rangle \in \mathcal{P}_{\dot{W}}$, because of $P \force \dot{\mathcal{B}}$ is maximal anti-chain below $\langle a,\emptyset \rangle$, there are $p' \leq p$, $n$, and, $P$-name $\dot{Z}$ such that $p' \force \{\langle b,Y \rangle \in \dot{\mathcal{B}} \mid |b| = n\}$ is maximal anti-chain below $\langle a,\dot{Z} \rangle \in \mathcal{P}_{\dot{U}}$. If $n \leq |a|$, there is a $\dot{Y}$ such that $p' \force  \langle b,\dot{Y} \rangle \in \dot{\mathcal{B}} \land a \setminus b \subseteq Y)$. Here, $b$ is the first $n$-th elements in $a$. Thus, it is forced that $\langle a,\dot{X} \cap \dot{Y} \rangle \leq \langle b,\dot{Y}\rangle, \langle a, \dot{X} \rangle$.

If $n > |a|$, we can choose $p'' \leq p'$ and $\alpha_0,...,\alpha_{n-|a|-1}$ with $p''\force \{\alpha_i \mid i < n-|a|\} \in [(\dot{X} \cap \dot{Z}) \setminus (\max{a} + 1)]^{n-|a|}$. Let $c = a \cup \{\alpha_{i}\mid i < n-|a|\}$. $p''$ forces that $\langle c,\dot{Z} \rangle \leq \langle a,\dot{Z}\rangle$ meet with $\dot{\mathcal{B}}$. Because of $|c| =n$, there is a $\dot{Y}$ with $p'' \force \langle c,\dot{Y} \rangle \in \dot{\mathcal{B}}$. In particular, it is forced that $\langle c,\dot{Y} \cap \dot{X}\rangle$ is a common extension of $\langle c,\dot{Y} \rangle$ and $\langle a,\dot{X} \rangle$, as desired.
\end{proof}

We introduce two forcing notions of variations of Prikry forcing. The first one is Woodin's modification~\cite{MR1007865}, which changes a measurable cardinal into $\aleph_{\omega}$. For a normal ultrafilter $U$ over $\mu$, let $j_{U}$ denote the ultrapower mapping $j_{U}:V \to M_U \simeq \mathrm{Ult}(V,U)$. If we suppose $2^{\mu} = \mu^{+}$ then $|j_{U}(\mu)| = 2^{\mu} = \mu^{+}$. This shows 
\begin{lem}\label{guidinggeneric}
 If $2^{\mu} = \mu^{+}$ then there is a $(M_U,\mathrm{Coll}(\mu^{+},<j_U(\mu))^{M_U})$-generic filter $\mathcal{G}$. 
\end{lem}
\begin{proof}
 Since $\mathrm{Coll}(\mu^{+},<j_U(\mu)))^{M_U}$ has the $j_U(\mu)$-c.c. in $M_{U}$ and $|j_U(\mu)^{<j_U(\mu)}| = |j_U(\mu)| = \mu^{+}$, we can enumerate $\mathrm{Coll}(\mu^{+},<j_U(\mu)))^{M_U}$ anti-chain belongs to $M_U$ as $\langle \mathcal{A}_\alpha \mid \alpha < \mu^{+} \rangle$. Because $\mathrm{Coll}(\mu^{+},<j_U(\mu)))^{M_U}$ is $\mu^{+}$-closed, the standard argument takes a filter $\mathcal{G}$ that meets with any $\mathcal{A}_{\alpha}$.
\end{proof}
We call this $\mathcal{G}$ a guiding generic of $U$. $\mathcal{P}_{U,\mathcal{G}}$ consists of $\langle a,f,X,F\rangle$ such that 
\begin{itemize}
 \item $a = \{\alpha_1,...,\alpha_{n-1}\} \in [\Psi]^{<\omega}$.
 \item $f = \langle f_0,...,f_{n-1}\rangle \in \prod_{i < n}\mathrm{Coll}(\alpha_i^{+},<\alpha_{i+1})$. But $\alpha_0$ and $\alpha_{n}$ denote $\omega$ and $\mu$, respectively.
 \item $X\in U$ and $X \subseteq \Psi$.
 \item $F \in \prod_{\alpha \in X} \mathrm{Coll}(\alpha^{+},<\mu)$ and $[F]\in \mathcal{G}$. 
\end{itemize}
Here,  $\Psi= \{\alpha < \mu\mid \alpha$ is an inaccessible and $2^{\alpha} = \alpha^{+}\}$.
$\mathcal{P}_{U,\mathcal{G}}$ is ordered by $\langle a,f,X,F\rangle \leq \langle b,g,Y,H\rangle$ if and only if $\langle a,X\rangle \leq \langle b,Y\rangle$ in $\mathcal{P}_{U}$, $\forall i \in [|b|,|a|)(h(i) \supseteq F(\beta_i))$, and $\forall \alpha \in X(F(\alpha)\supseteq H(\alpha))$. It is easy to see that $\mathcal{P}_{U,\mathcal{G}}$ is $(\mu,<\mu)$-centered.

Lemma \ref{modificationprikrylemma} and \ref{modificationprikrycondi} are analogies of Lemma \ref{prikrylem} and \ref{prikrycondi} for $\mathcal{P}_{U,\mathcal{G}}$, respectively.
\begin{lem}\label{modificationprikrylemma}
 For any $\langle a,f,X,F\rangle$ and $\sigma$, there is a $\langle a,f,Z,I\rangle$ such that, if $\langle b,g,Y,G\rangle\leq \langle a,f,Z,I\rangle$ decides $\sigma$ then $\langle a,g\upharpoonright |a|,Z,I\rangle$ decides $\sigma$.
\end{lem} 
This shows that the cardinality of $\mu$ is preserved by $\mathcal{P}_{U,\mathcal{G}}$. By the density argument shows $\mathcal{P}_{U,\mathcal{G}} \force \mu = \aleph_{\omega}$.
\begin{lem}\label{modificationprikrycondi}
 For any $\langle a,f,X,F\rangle$ and maximal anti-chain $\mathcal{A}$ below $p$, there are $n,f',Z,I$ such that $\{\langle b,g,Y,H\rangle \in \mathcal{A} \mid |b| = n\}$ is a maximal anti-chain below $\langle a,f',Z,I\rangle$. 
\end{lem}

The other is Magidor forcing~\cite{Magidor1978changing}. Magidor forcing used a sequence of normal ultrafilters over $\mu$ instead of a single normal ultrafilter. For the definition of Magidor forcing and its details, we refer to~\cite{Magidor1978changing} or ~\cite{MR4404936}.

\begin{thm}[Magidor~\cite{Magidor1978changing}]\label{magidorforcing}
Suppose $\mu$ is supercompact and $\nu < \mu$ is regular. Then there is a poset $P$ such that
\begin{itemize}
 \item $P$ is $(\mu,<\mu)$-centered.
 \item $P$ adds no new subset to $\nu$. Thus, the regularities below $\nu$ are preserved.
 \item $P$ preserves all cardinals.
 \item $P \force \mathrm{cf}(\mu) = \nu$.
\end{itemize} 
\end{thm}

\subsection{Combinatorics}
The notion of polarized partition relations was introduced by Erd\H{o}s--Hajnal--Rado~\cite{MR202613}. $\polar{\kappa_0}{\kappa_1}{\lambda_0}{\lambda_1}{\theta}$ states, for every $f:\kappa_0 \times \kappa_1 \to \theta$ there are $H_0 \in [\kappa_0]^{\lambda_0}$ and $H_1 \in [\kappa_1]^{\lambda_1}$ such that $|f ``H_0 \times H_1| \leq 1$. $\polar{\kappa_0}{\kappa_1}{\kappa_0}{\kappa_1}{\theta}$ is the most strongest form. This form trivially holds sometimes. Indeed, if $\mathrm{cf}(\kappa_0) > \theta^{\kappa_1}$ then $\polar{\kappa_0}{\kappa_1}{\kappa_0}{\kappa_1}{\theta}$ holds. But under the GCH, the non-trivial case cannot hold:
\begin{thm}[Erd\H{o}s--Hajnal--Rado~\cite{MR202613}]
 If $2^{\mu} = \mu^{+}$ then $\polar{\mu^{+}}{\mu}{\mu^{+}}{\mu}{2}$. 
\end{thm}
We are interested in how strong $\polar{\mu^{+}}{\mu}{\lambda_0}{\lambda_1}{\theta}$ can hold under the GCH. If $\mu$ is a limit cardinal, $\polar{\mu^{+}}{\mu}{\mu}{\mu}{<\mu}$ holds sometimes (For example, see \cite{MR1833480}, \cite{MR1606515}, and \cite{MR0371655}). On the other hand, for a successor cardinal, negative partition relation is known as Theorem \ref{kurepaimpliesnpp}. 
\begin{thm}[Folklore?]\label{kurepaimpliesnpp}
 If there is a Kurepa tree on $\mu^{+}$ then $\npolar{\mu^{++}}{\mu^{+}}{2}{\mu^{+}}{\mu}$ holds.
\end{thm}
Therefore $\polar{\mu^{++}}{\mu^{+}}{2}{\mu^{+}}{\mu}$ is a large cardinal property. 
Erd\H{o}s--Hajnal~\cite{MR0280381} asked whether or not $\polar{\aleph_2}{\aleph_1}{\aleph_0}{\aleph_1}{2}$ is consistent. To solve this, the notion of a strongly saturated ideal was introduced by Laver. Laver also proved that 
\begin{thm}[Laver~\cite{MR673792}]\label{stronglysatimplypp}
 Suppose that $2^{\mu} = \mu^{+}$ and $\mu^{+}$ carries a strongly saturated ideal. Then $\polar{\mu^{++}}{\mu^{+}}{\mu^{+}}{\mu^{+}}{\mu}$ holds. 
\end{thm}
The assumption of Theorem \ref{stronglysatimplypp} implies $2^{\mu^+} = \mu^{++}$, and thus, $\npolar{\mu^{++}}{\mu^{+}}{\mu^{++}}{\mu^{+}}{2}$. We will discuss $\polar{\mu^{++}}{\mu^{+}}{\nu}{\mu^{+}}{\mu}$ for $\nu \in [2,\mu^{+}]$ in Section \ref{ccandpp}.

We recall basic properties of Chang's conjecture for section 4 and showing Theorem \ref{maintheorem3}.(5).

For cardinals $\lambda \geq \lambda'$ and $\kappa \geq \kappa' \geq \mu$, we need a strengthening of Chang's conjecture $(\lambda,\lambda') \twoheadrightarrow_{\mu} (\kappa,\kappa')$, which was introduced by Shelah~\cite{MR1126352}. We say $(\lambda,\lambda') \twoheadrightarrow_{\mu} (\kappa,\kappa')$ if any structure $\langle \lambda,\lambda',\in,...\rangle$, of a language of size $\mu$ has an elementary substructure $\langle X,X \cap \lambda',\in,...\rangle$ such that $|X| = \kappa$, $|X \cap \lambda'| = \kappa'$, and $\mu \subseteq X$. Note that $(\lambda,\lambda') \twoheadrightarrow_{\omega} (\kappa,\kappa')$ is the same as $(\lambda,\lambda') \twoheadrightarrow (\kappa,\kappa')$. It is easy to see 
\begin{lem}\label{changchar}
 The following are equivalent:
\begin{enumerate}
 \item $(\lambda,\lambda') \twoheadrightarrow_{\mu} (\kappa,\kappa')$ holds.
 \item Any structure of a countable language $\langle \lambda,\lambda',\in,...\rangle$ has a substructure $\langle X,X \cap \lambda',\in,...\rangle$ such that $|X| = \kappa$, $|X \cap \lambda'| = \kappa'$, and $\mu \subseteq X$.
 \item For every $f:{^{<\omega}}\lambda \to \lambda$, there is an $X \in [\lambda]^{\kappa}$ such that $X$ closed under $f$, $|X \cap \lambda'| = \kappa'$, and $\mu \subseteq X$. 
\end{enumerate}
\end{lem}
\begin{proof}
$(1) \to(2) \to (3)$ is easy. We check $(3) \to (1)$.

For any structure $\mathcal{A} = \langle \lambda,\lambda',\in,...\rangle$ of a language of size $\mu$, there is a complete set of Skolem functions $\{f_{\xi} \mid \xi < \mu\}$. Define $g:{^{<\omega}\lambda} \to \lambda$ by 
\begin{center}
 $g(a) = \begin{cases}f_{a}(b) & a = \langle \alpha,b\rangle\text{ for some }\alpha < \mu\\
0 & \text{otherwise}
	 \end{cases}.$\end{center}
By the assumption, there is an $X \in [\lambda]^{\kappa}$ such that $X$ closed under $g$, $|X \cap \lambda'| = \kappa'$, and $\mu \subseteq X$. Then $\langle X,X \cap \kappa ,\in ,... \rangle \prec \mathcal{A}$ witnesses.
\end{proof}
Chang's conjecture follows by the existence of elementary embeddings. 
\begin{lem}\label{changsuff}
 Suppose that $j$ is an elementary embedding from $V$ to $M$ which is defined in an outer model. For cardinals $\lambda \geq \lambda'$ and $\kappa \geq \kappa' \geq \mu$ in $V$, suppose that 
\begin{itemize}
 \item $\mathrm{crit}(j) > \mu$, 
 \item $j``\lambda \in M$,
 \item $j(\kappa) = |j``\lambda|$ in $M$, and
 \item $j(\kappa') = |j``\lambda'|$ in $M$.
\end{itemize}
 Then $(\lambda,\lambda') \twoheadrightarrow_{\mu} (\kappa,\kappa')$.
\end{lem}
\begin{proof}
 For any $\mathcal{A} = \langle \lambda,\lambda',\in,...\rangle$, $\mathcal{B} = \langle j``\lambda,j``\lambda',\in,... \rangle$ is a substructure of $\mathcal{A}$ that witnesses with $(\lambda,\lambda') \twoheadrightarrow_{\mu} (\kappa,\kappa')$. 
\end{proof}

\begin{lem}[Folklore]\label{ccpreserved}
 $(\lambda,\lambda') \twoheadrightarrow_{\mu} (\kappa,\kappa')$ is preserved by $\mu^{+}$-c.c. poset. 
\end{lem}
\begin{proof}This proof is due to Eskew--Hayut~\cite{MR3748588}.
 Let $P$ be a $\mu^{+}$-c.c. poset. Assume $(\lambda,\lambda') \twoheadrightarrow_{\mu} (\kappa,\kappa')$. For each $p \force \dot{f} :{^{<\omega}}\lambda \to \lambda$ and $a \in {^{<\omega}}\lambda$, by the $\mu^{+}$-c.c., there is an $X_{a} \in [\lambda]^{\leq\mu}$ such that $p \force \dot{f}(a) \in X_{a}$. Define $g(a)$ by
\begin{center}
 $g(a) = \begin{cases}\text{the }\alpha\text{-th element in }X_{b} & a = \langle \alpha,b\rangle\text{ for some }\alpha < \mu\\
0 & \text{otherwise}
	  
	 \end{cases}.$
\end{center}
 By Lemma \ref{changchar}, we have an $X \in [\lambda]^{\kappa}$ such that $|X \cap \lambda'| = \kappa'$ and $\mu \subseteq X$. Note that each $X_{a}$ is of size $\leq \mu$ and $\mu \subseteq X$. For every $a \in {^{<\omega}X}$, we have 
\begin{center}
 $p \force \dot{f}(a) \in X_{a} = \{g(\alpha,a) \mid \alpha < \mu\} \subseteq g``({^{<\omega}}X) \subseteq X$. 
\end{center}
By Lemma \ref{changchar}, the proof is completed.
\end{proof}

\section{Centeredness and Layeredness}\label{centeredlayered}
In this section, we show Theorem \ref{maintheorem1}. Lemma \ref{termcentered} is essentially due to Foreman. We will show a more general result as Lemma \ref{generalizedpreservation}. Our proof of Lemma \ref{termcentered} is a prototype of that of Lemma \ref{generalizedpreservation}.

To study the centeredness, we use the notion of the term forcing. For a poset $P$ and a $P$-name $\dot{Q}$ for a poset, the term forcing $T(P,\dot{Q})$ is a complete set of representatives from $\{\dot{q} \mid  \force \dot{q} \in \dot{Q}\}$ with respect to the canonical equivalence relation. $T(P,\dot{Q})$ is ordered by $\dot{q} \leq \dot{q}' \leftrightarrow \force \dot{q} \leq \dot{q}'$. The following is known as the basic lemma of the term forcing.

\begin{lem}[Laver]\label{laverbasiclemma}
 $\mathrm{id}:P \times T(P,\dot{Q}) \to P \ast \dot{Q}$ is a projection. In particular, $P \force$ there is a projection from $T(P,\dot{Q})$ to $\dot{Q}$. 
\end{lem}

\begin{lem}[Foreman]\label{termcentered}
 Suppose that $P$ is $(\mu,<\nu)$-centered and $I$ is a $(\mu^{+},<\nu)$-centered ideal over $\mu^{+}$. If $2^{\mu} = \mu^{+}$ then $T(P,\dot{\mathcal{P}}(\mu^{+})/\overline{I})$ is $(\nu^{+},<\nu)$-centered. 

 In particular, if $P$ is $\nu$-Baire then $P$ forces that $\overline{I}$ is $(\mu,<\nu)$-centered.
\end{lem}
\begin{proof}
 Let $f:I^{+} \to \mu^{+}$ be a $(\mu^{+},<\nu)$-centering function. Let $\{P_{\alpha} \mid \alpha < \mu\}$ be a $(\mu,<\nu)$-centering family of $P$. We may assume that each $P_\alpha$ is a $<\nu$-complete filter. 

 We want to define a $(\mu^{+},<\nu)$-centering function $h:T(P,\dot{\mathcal{P}}(\mu^{+})/\overline{I})) \to \mu^{+}$. For each $\dot{A} \in T(P,\dot{\mathcal{P}}(\mu^{+})/\overline{I}))$, $B = \{\xi < \mu^{+} \mid ||\xi \in \dot{A} || \not=0\} \in I^{+}$. For each $\alpha < \nu$, if we let $B_{\alpha} = \{\xi < \mu^{+} \mid ||\xi \in \dot{A}|| \in P_{\alpha}\}$ then $B = \bigcup_{\alpha}B_{\alpha}$. Since $I$ is $\mu^{+}$-complete, there is an $\alpha < \mu$ with $B_{\alpha} \in I^{+}$.

Define $h(\dot{A})$ by $\langle f(B_{\alpha}) \mid \alpha < \mu, B_{\alpha} \in I^{+}\rangle$. Note that the range of $h$ is ${^{\leq\mu}}\mu^{+}$. Therefore $h$ can be seen as a mapping with its range $\mu^{+}$. 

 For $\{\dot{A}_{i} \mid i < \nu'\} \in [T(P,\dot{\mathcal{P}}(\mu^{+})/\overline{I}))]^{<\nu}$, if $h(\dot{A}_i) = d$ for all $i < \nu'$. For each $i$ and $\alpha < \mu$, $B_{\alpha}^{i} = \{\xi < \mu \mid \exists q \in P_{\alpha} (q \force \xi \in \dot{A}_{i})\}$. We want to show $P \force \bigcap_{i} \dot{A}_i \in \overline{I}^{+}$.

 \begin{clam}\label{kanamorilemma}
  There is an $A \subseteq \mu^{+}$ such that $A \in I$ and $P \force \bigcap_{i} \dot{A}_{i} \setminus A \not= \emptyset \to \bigcap_{i} \dot{A}_{i} \in \overline{I}^{+}$. 
 \end{clam} 
 \begin{proof}[Proof of Claim]
  Our proof is based on the proof in \cite[Theorem 17.1]{MR1994835}. Let $\mathcal{A} \subseteq P$ be a maximal subset such that 
\begin{itemize}
 \item $\mathcal{A}$ is an anti-chain.
 \item $\forall p\in \mathcal{A} \exists A_{p}\subseteq \mu^{+}(\mu^{+} \setminus A_p \in I \land p \force \bigcap_{i}\dot{A} \subseteq A_p)$.
\end{itemize}
 We note $\sum \mathcal{A} = ||\bigcap_{i}\dot{A}_i \not\in \overline{I}^{+}||$. By the $\mu^{+}$-c.c. of $P$, $|\mathcal{A}| \leq \mu$. Let $A = \bigcup_{p \in \mathcal{A}_p} A_p$. By the $\mu^{+}$-completeness of $I$, $A \in I$. For every $p \in P$, if $p \force \bigcap_i \dot{A}_i \setminus A \not= \emptyset$ then $p$ and $\sum \mathcal{A}$ are incompatible, and thus, $p$ forces $\bigcap_i \dot{A}_i \in \overline{I}^{+}$. 
 \end{proof}
 For each $p \in P$ and $j$, there is an $\alpha < \mu$ such that $B_{\alpha}^j \in {I}^{+}$ and $p \in P_{\alpha}$. By the assumption, $h(B_{\alpha}^{i}) = h(B_\alpha^{j})$ for all $i <j< \nu'$. Since $h$ is $(\mu^{+},<\nu)$-centering, $\bigcap_{i}B_{\alpha}^{i} \in I^{+}$. We can choose $\xi \in \bigcap_{i}B_{\alpha}^{i} \setminus A$. By the definition of $B_{\alpha}^{i}$, there exists $q_{i} \in P_{\alpha}$ which forces $\xi \in \dot{A}_{i} \setminus A$ for each $i < \nu'$. Since $P_{\alpha}$ is a $<\nu$-complete filter, $q := p \cdot \prod_{i}q_{i} \in P_{\alpha}$. $q \leq p$ forces $\xi \in \bigcap_{i}\dot{A}_i \setminus A$. By the claim, $q \force \bigcap_i \dot{A}_i \in \overline{I}^{+}$, as desired.

 If $P$ is $\nu$-Baire, then $P \force T(P,\dot{\mathcal{P}}(\mu^{+})/\overline{I}))$ remains $(\mu^{+},<\nu)$-centered. By Lemma \ref{laverbasiclemma}, $P \force \overline{I}$ is $(\mu^{+},<\nu)$-centered, as desired.
\end{proof}
Next, we deal with layeredness. Let us describe a sufficient condition for the quotient forcing \emph{not} to be $S$-layered. We say that $Q$ is nowhere $S$-layered if $Q\upharpoonright q$ is not $S$-layered for all $q \in Q$. 

\begin{lem}\label{quotientnotlayered}
 Suppose that $Q$ is nowhere $S$-layered for some $S \subseteq E^{\mu^{++}}_{\mu^{+}}$, and $Q$ is of size $\mu^{++}$. We also assume that there is a complete embedding $\tau$ from $\mu^{+}$-c.c. $P$ to $Q$. Then $P \force Q / \dot{G}$ is not $S$-layered.
\end{lem}
\begin{proof}
 Suppose otherwise. That is, there is a $p \in P$ which forces that $Q/ \dot{G}$ is $S$-layered. By the assumption and Lemma \ref{charlayered}, we can fix $P$-names $\dot{R}_{\alpha}$ such that 
\begin{itemize}
 \item $p \force \dot{R}_{\alpha}\lessdot Q / \dot{G}$ for each $\alpha$. 
 \item $p \force \alpha < \beta \to \dot{R}_{\alpha}\subseteq \dot{R}_{\beta}$. 
 \item $p \force$ there is a club $C \subseteq \mu^{++}$ such that $\forall \alpha \in C \cap S(\dot{R}_{\alpha} = \bigcup_{\beta < \alpha} \dot{R}_{\beta})$. 
\end{itemize}
 By the $\mu^{+}$-c.c. of $P$, we can choose such a club $C$ in $V$.

 We claim that $P \upharpoonright p \ast (Q / \dot{G})$ is $S$-layered. Let $Q_{\alpha} = P \upharpoonright p \ast \dot{R}_{\alpha}$. It is easy to see that $Q_{\alpha}\lessdot P \upharpoonright p \ast (Q / \dot{G})$. For $\alpha \in C \cap S$, choose $\langle p_0,\dot{q}_0\rangle \in Q_{\alpha}$ then $p \force \dot{q}_{0} \in \dot{R}_{\alpha}$. By $\mathrm{cf}(\alpha) = \mu^{+}$ and $P$ has the $\mu^{+}$-c.c., there is an $\beta < \alpha$ such that $P \force \dot{q}_{0} \in P_{\beta}$. Therefore $\langle p_0,\dot{q}_0\rangle \in R_{\beta}$, as desired.

 Since $\mathcal{B}(Q)$ has a dense subset that is isomorphic to $P \ast (Q / \dot{G})$ and $P \ast (Q / \dot{G}) \upharpoonright \langle p,\dot{1}\rangle$ is $S$-layered, this contradicts that $Q$ is nowhere $S$-layered. 
\end{proof}
To show Theorem \ref{maintheorem1}, the following is a key lemma.

\begin{lem}\label{mainlemmalayered}
Suppose that $I$ is a saturated ideal over $\mu^{+}$ and $P$ is one of Prikry forcing, Woodin's modification, or Magidor forcing at $\mu$. Then ${\mathcal{P}}(\mu^{+})/{I} \ast \dot{j}(P)$ is nowhere $S$-layered for all stationary $S \subseteq \mu^{++}$. 
\end{lem}
\begin{proof}
 Because the similar proof works for each of Prikry-type forcings, we only check if $P = \mathcal{P}_{U}$ for some normal ultrafilter $U$ over $\mu$.
 If $\mathcal{P}(\mu^{+})/I$ is nowhere $S$-layered, there is nothing to do. We assume that there is an $A \in I^{+}$ with $I \upharpoonright A$ is $S$-layered. For simplicity, we assume that $I$ is $S$-layered. 

We fix sufficiently large regular $\theta$ and $M \prec \mathcal{H}_{\theta}$ such that
\begin{itemize}
 \item $|M| = \mu^{+}$ and $\mu \subseteq M$,
 \item $\mathcal{P}(\mu^{+})/I \cap M \lessdot \mathcal{P}(\mu^{+})/I$ forces $|(\mu^{+})^{V}| = \mu$, and
 \item $M$ contains all relevant elements.
\end{itemize}
It is enough to show that $Q \cap M \not\mathrel{\lessdot} Q$. Let $\dot{F}$ be a $\mathcal{P}(\mu^{+})/I \cap M$-name for the filter generated by $\{X \in M[\dot{G}] \mid \exists q \in G(q \force X \in \dot{j}({U})\}$. It is easy to see that $Q \cap M$ is dense in $(\mathcal{P}(\mu^{+})/I \cap M) \ast \mathcal{P}_{\dot{F}}$. By $\mathcal{P}(\mu^{+})/I \cap M \force |{M}[\dot{G}]| = |M| \leq |(\mu^{+})^{V}| = \mu$, $\mathcal{P}(\mu^{+})/I \cap M \force \dot{F}$ is not an ultrafilter. By Lemma \ref{prikryequiv}, we have
 $Q \cap M \simeq (\mathcal{P}(\mu^{+})/I \cap M) \ast \mathcal{P}_{\dot{F}} \not\mathrel{\lessdot} \mathcal{P}(\mu^{+})/I \ast \mathcal{P}_{\dot{j}(U)}$. 
\end{proof}

Let us show Theorem \ref{maintheorem1}.  
\begin{proof}[Proof of Theorem \ref{maintheorem1}]
 Let $P$ be one of Prikry forcing, Woodin's modification, or Magidor forcing. Then $P$ is $\mu$-centered.  For (1), by Lemma \ref{termcentered}, $P \force \overline{I}$ is $\mu$-centered. 

Let us show (2). Since $\mu^{+}$ carries a saturated ideal and $2^{\mu}= \mu^{+}$, we know $2^{\mu^{+}} = \mu^{++}$. Note that $\dot{j}(P)$ is forced to be $\mu$-centered by $\mathcal{P}(\mu^{+}) / I$. Therefore $\mathcal{P}(\mu^{+})/I \ast \dot{j}(P)$ is of size $2^{\mu^{+}} = \mu^{++}$ since $\mathcal{P}(\mu^{+})/I \ast \dot{j}(P)$ is $\mu^{+}$-centered. 

 The size of $P$ is $2^{\mu} =\mu^{+}$ by the $\mu$-centeredness of $P$. Fix a $P$-name for a stationary subset $\dot{T}$ such that $P \force \dot{T} \subseteq \dot{E}^{\mu^{++}}_{\mu^{+}}$. Let $S_q = \{\alpha < \mu^{++} \mid q \force \alpha \in \dot{T}\}$. Then $P \force \dot{T} = \bigcup_{q \in \dot{G}} S_{q}$. This implies that there is $q$ such that $q \force S_q \subseteq \dot{T}$ is stationary.  Since $P$ has the $\mu^{+}$-c.c., $P$ does not change the set $E^{\mu^{++}}_{\mu^{+}}$, therefore $S_q \subseteq E_{\mu^{+}}^{\mu^{++}}$. By Lemma \ref{quotientnotlayered}, $P \force Q / \dot{G}$ is not $S_q$-layered. In particular, $q \force Q / \dot{G}$ is not $\dot{T}$-layered. Since $q$ and $\dot{T}$ are arbitrary, $\force Q / \dot{G}$ is not $\dot{T}$-layered for all stationary $\dot{T} \subseteq \dot{E}^{\mu^{++}}_{\mu^{+}}$. 

By Theorem \ref{duality}, $P \force \mathcal{P}(\mu^{+}) / \overline{I} \simeq \mathcal{B}(Q / \dot{G})$ is not $\dot{T}$-layered for all stationary $\dot{T} \subseteq \dot{E}^{\mu^{++}}_{\mu^{+}}$. That is, $P \force \overline{I}$ is \emph{not} layered.
\end{proof}

We note that the proof of Lemma \ref{foremanfst}, which was given in~\cite{MR730584}, shows that
\begin{lem}\label{saturationinprikry}
  Suppose that $\mu$ is a regular cardinal and $I$ is a $(\mu^{++},\nu,\nu)$-saturated ideal over $\mu^+$. Then every $(\mu,<\nu^{+})$-centered poset forces that $\overline{I}$ is $(\mu^{++},\nu,\nu)$-saturated.
\end{lem}
\begin{proof}
 Let $P$ be a $(\mu,<\nu^{+})$-centered poset. We may assume that $P$ is a complete Boolean algebra. Let $\{P_{\alpha} \mid \alpha < \lambda\}$ be $\nu^{+}$-complete filters that cover $P$. 

Let $p \in P$ and $\{\dot{A}_{i} \mid i < \mu^{++}\}$ be arbitrary such that $p \force \dot{A}_{i}\in [I^{+}]^{\mu^{++}}$ for each $i$. As proof of Theorem \ref{duality}, we let consider the set $B_{i} = \{\xi < \mu^{+} \mid p\cdot ||\xi \in \dot{A}_{i}||\not=0\} \in I^{+}$. Define $B^i_{\alpha} = \{\xi \in B_i \mid p \cdot ||\xi \in \dot{A}_{i}|| \in P_{\alpha}\}$. By $B_{i} = \bigcup_{\alpha < \mu}B_{\alpha}^i$, there is an $\alpha_i < \mu$ such that $B_{\alpha_i}^i \in I^{+}$. There are $K \in [\mu^{++}]^{\mu^{++}}$ and $\alpha$ such that $\forall i \in K(\alpha_i = \alpha)$. Since $I$ is $(\mu^{++},\nu,\nu)$-saturated, there is $Z \in [K]^{\nu}$ such that $\bigcap_{i \in Z}B_\alpha^{i} \in I^{+}$. 

By the proof of Claim \ref{kanamorilemma}, we have $A \in I$ such that $P \force \bigcap_{i\in Z} \dot{A}_i \setminus A \not= 0$ implies $\bigcap_i \in \overline{I}^{+}$. Then, For each $\xi \in \bigcap_{i \in Z}B_\alpha^i \in I^{+} \setminus A$, $q:=\prod_{\xi \in Z} p\cdot ||\xi \in \dot{A}_i|| = p \cdot ||\xi \in \bigcap_{i \in Z}\dot{A}_i||\not= 0$ forces $\bigcap_{i\in Z} \dot{A}_i \setminus A \not= 0$.  By the choice of $A$, $q \force \bigcap_{i \in Z} \dot{A}_i \in \overline{I}^{+}$.
\end{proof}
We notice that, by the proof of the Theorem \ref{duality}, this proof of Lemma \ref{saturationinprikry} essentially shows that $\mathcal{P}(\mu^{+}) / I \ast \dot{j}(P)$ has the $(\mu^{++},\nu,\nu)$-c.c. Note that if $Q$ has the $(\mu^{++},\nu,\nu)$-c.c. and there is a complete embedding $e:P \to Q$ then $P \force Q / \dot{G}$ has the $(\mu^{++},\nu,\nu)$-c.c. Thus $P \force \mathcal{P}(\mu^{+}) / I \ast \dot{j}(P)/ \dot{G}$ has the $(\mu^{++},\nu,\nu)$-c.c.

We conclude this section with the following question.
\begin{ques}
 Can a layered ideal exist over the successor of singular cardinals?
\end{ques}

Eskew pointed out that
\begin{thm}[Eskew--Sakai~\cite{MR4092254}]
 There is no dense ideal over the successor of singular cardinals.
\end{thm}
Here, a dense ideal over $\mu^{+}$ is an ideal $I$ such that $\mathcal{P}(\mu^{+}) / I$ has a dense subset of size $\mu^{+}$. Density is the strongest saturation property. The consistency of dense ideal on $\omega_1$ is known by Woodin. After that, Eskew~\cite{MR3569105} extended Woodin's result to an arbitrary successor of regular cardinals. 

\section{Polarized Partition Relations}\label{ccandpp}
In the first half of this section, we collect some sufficient conditions for polarized partition relations in terms of Chang's conjecture and the existence of some saturated ideals. The rest of this section is devoted to proving Theorem \ref{maintheorem2} and its application. Lemma \ref{ppmainlemma2}, which is used in a proof of Theorem \ref{maintheorem2}, answers \cite[Question 1.11]{MR4101445}. 

In Section 2, we saw Laver's result for strongly saturated ideals and polarized partition relations in Theorem \ref{stronglysatimplypp}. Other than this result, it is known that Todor\v{c}evi\'{c}~\cite{MR1127033} proved that $(\omega_2,\omega_1) \twoheadrightarrow (\omega_1,\omega)$ implies $\polar{\omega_2}{\omega_2}{\omega}{\omega}{\omega}$. Modifying the proof of Todor\v{c}evi\'{c}, Zhang~\cite{MR4094551} showed that the existence of pre-saturated ideal over $\omega_1$ implies $\polar{\omega_2}{\omega_1}{\omega}{\omega}{\omega}$ and $\polar{\omega_2}{\omega_1}{n}{\omega_1}{\omega}$ for any $n < \omega$. Here, pre-saturated ideal over $\mu^{+}$ is a precipitous ideal $I$ with $\mathcal{P}(\mu^{+})/ I$ preserves the cardinality of $\mu^{++}$. Let us generalize this
\begin{lem}\label{ccimpliespp}
 Assume one of the following holds:
 \begin{enumerate}
  \item $(\mu^{++},\mu^{+})\twoheadrightarrow (\mu^{+},\mu)$.
  \item $\mu^{+}$ carries a pre-saturated ideal.
 \end{enumerate}
Then $\polar{\mu^{++}}{\mu^{+}}{n}{\mu^{+}}{\mu}$ holds for all $n < \omega$. 
\end{lem}
\begin{proof}
 Let $f:\mu^{++} \times \mu^{+} \to \mu$ be an arbitrary coloring. 

 First, we assume $(\mu^{++},\mu^{+})\twoheadrightarrow (\mu^{+},\mu)$. By \cite[Lemma 14]{MR3748588}, we have $(\mu^{++},\mu^{+})\twoheadrightarrow_{\mu} (\mu^{+},\mu)$. Consider a structure $\mathcal{A} = \langle \mu^{++},\mu^{+},\in,f\rangle$. By Lemma \ref{changchar}, we can choose a $\mathcal{B} = \langle X,X \cap \mu^{+},\in,f \upharpoonright X\rangle \prec \mathcal{A}$ such that $|X| = \mu^{+}$, $|X \cap \mu^{+}| = \mu$, and $\mu \subseteq X$. Let $\delta = \sup X \in \mu^{+}$. There are $\alpha_0,...,\alpha_{n-1} \in X$ such that $f(\alpha_0,\delta)= \cdots = f(\alpha_{n-1},\delta)= \eta$ for some $\eta \in \mu \subseteq X$. 

Then, the elementarity shows
\begin{center}
 $\mathcal{B} \models \forall \xi \in \mu^{+}\exists \zeta \geq \xi(f(\alpha_0,\zeta)\land \cdots \land f(\alpha_{n-1},\zeta) = \eta)$. 
\end{center}
Indeed, for every $\xi \in X$, $\zeta\geq \xi$ can be taken as $\delta$ in $\mathcal{A}$. In particular, $H_1 = \{\xi <\mu^{+}\mid \forall i < n(f(\alpha_i,\xi) = \eta)\}$ is unbounded in $\mu^{+}$. $f ``\{\alpha_0,...,\alpha_{n-1}\} \times H_1 = \{\eta\}$.

 Next, we assume the existence of pre-saturated ideal $I$. Let $G$ be a $(V,\mathcal{P}(\mu^{+}) / I)$-generic filter and $j:V \to M \subseteq V[G]$ be a generic ultrapower mapping. Note that $|j ``\mu^{++}| = (\mu^{+})^{V[G]}$ and $\mathrm{crit}(j) = (\mu^{+})^V$. Then, in $V[G]$, there are $\alpha_0,...,\alpha_{n-1} \in \mu^{++}$ and $\eta$ such that $M \models j(f)(j(\alpha_0),(\mu^{+})^{V}) = \cdots =j(f)(j(\alpha_{n-1}),(\mu^{+})^{V}) = \eta = j(\eta)$. 

Again, the elementarity shows that $H_1 = \{\xi <\mu^{+}\mid \forall i < n(f(\alpha_i,\xi) = \eta)\}$ is unbounded in $\mu^{+}$ and $f ``\{\alpha_0,...,\alpha_{n-1}\} \times H_1 = \{\eta\}$. 
\end{proof}
For each $\nu \leq \mu$, if $\mu$ carries a pre-saturated ideal $I$ that satisfies the condition of $\mathcal{P}(\mu^{+}) / I \force \forall f:\mu^{++} \to {\mathrm{ON}} \exists X \in [\mu^{++}]^{\nu}(f \upharpoonright X \in V)$, then the same proof of Theorem \ref{ccimpliespp} shows $\polar{\mu^{++}}{\mu^{+}}{\nu}{\mu^{+}}{\mu}$. It is easy to see that this condition is equivalent with the $(\mu^{++},\nu,\nu)$-saturation for saturated ideals. Note that we can show the following without using generic ultrapower:

\begin{lem}\label{saturatedimplypolarized}
If $\mu^{+}$ carries $(\mu^{++},\nu,\nu)$-saturated ideal for some $\nu \leq \mu$ then $\polar{\mu^{++}}{\mu^{+}}{\nu}{\mu^{+}}{\mu}$ holds.
\end{lem}

\begin{proof}
 Let $I$ be a $(\mu^{++},\nu,\nu)$-saturated ideal. 

Let $f:\mu^{++} \times \mu^{+} \to \mu$ be an arbitrary coloring. For each $\alpha < \mu^{++}$, there is an $\eta_{\alpha}$ such that $A_{\alpha} = \{\xi < \mu^{+} \mid f(\alpha,\xi) = \eta_{\alpha}\} \in F^{+}$. By the $(\mu^{++},\nu,\nu)$-saturation and $\mu^{+}$-completeness of $I$, there are $H_0 \in [\mu^{++}]^{\nu}$ and $\eta$ such that $H_1 = \bigcap_{\alpha \in H_0}A_{\alpha} \in F^{+}$ and $\forall \alpha \in H_0(\eta_{\alpha} = \eta)$. Then $f ``H_0 \times H_1 = \{\eta\}$. 
\end{proof}
The existence of $(\mu^{++},\nu,\nu)$-saturated ideals is preserved by $(\mu^{+},<\nu^{+})$-centered poset as we saw in Lemma \ref{saturationinprikry}. Therefore, under the existence of this ideal, $\polar{\mu^{++}}{\mu^{+}}{\nu}{\mu^{+}}{\mu}$ is also preserved by $(\mu^{+},<\nu^{+})$-centered posets. We can omit the ideal assumption from this fact (See Lemma \ref{ppmainlemma1}).

Lemmas \ref{ppmainlemma1}, \ref{ppmainlemma2}, \ref{negationpreserved}, and \ref{ppmainlemma3} show Theorem \ref{maintheorem2}. Let us begin to prove these.

\begin{lem}\label{ppmainlemma1}
 If $\polar{\mu^{++}}{\mu^{+}}{\nu}{\mu^{+}}{\mu}$ holds for some cardinal $\nu < \mu$ then any $(\mu,<\nu^{+})$-centered poset forces the same partition relation.
\end{lem}
\begin{proof}
 Let $P$ be a $(\mu,<\nu^{+})$-centered poset and $\langle P_{\alpha}\mid \alpha < \mu\rangle$ be a $(\mu,<\nu^{+})$-centering family of $P$. Consider $p \force \dot{f}:\mu^{++} \times \mu^{+} \to \mu$. For each $\alpha \in \mu^{++}$ and $\xi < \mu^{+}$, we can choose $p_{\alpha\xi}\leq p$ and $\eta_{\alpha\xi}$ such that $p_{\alpha\xi} \force \dot{f}(\alpha,\xi) = \eta_{\alpha\xi}$. We fix $\beta_{\alpha\xi} < \mu$ with $p_{\alpha\xi} \in P_{\beta_{\alpha\xi}}$

Define $d(\alpha,\xi) = \langle \beta_{\alpha\xi},\eta_{\alpha\xi} \rangle$. By $\polar{\mu^{++}}{\mu^{+}}{\nu}{\mu^{+}}{\mu}$, there are $H_0 \in [\mu^{++}]^{\nu}$ and $H_1 \in [\mu^{+}]^{\mu^{+}}$ such that $d$ is monochromatic on $H_0 \times H_1$ with value $\langle \beta,\eta\rangle$. 

For each $\xi \in H_1$, $q_{\xi}$ be a lower bound of $\{p_{\alpha\xi} \mid \alpha \in H_0\}$.  By the $\mu^{+}$-c.c. of $P$, there is a $q \leq p$ which forces that $|\{\xi \in H_1\mid q_{\xi} \in \dot{G}\}| = \mu^{+}$. Let $\dot{H}_1$ be a $P$-name for such set. We have $q \force \dot{f} ``H_{0}\times \dot{H}_{1} = \{\eta\}$.
\end{proof}

\begin{lem}\label{ppmainlemma2}
 If $\npolar{\mu^{++}}{\mu^{+}}{n}{\mu^{+}}{\mu}$ for some regular $n < \omega$ then any $\mu^{+}$-Knaster poset forces the same partition relation.
\end{lem}
\begin{proof} 
 Let $f$ be a coloring that witnesses with $\npolar{\mu^{++}}{\mu^{+}}{n}{\mu^{+}}{\mu}$. For each $p \force \dot{H}_{1} \in [\mu^{+}]^{\mu^{+}}$ and $H_0 \in [\mu^{++}]^{n}$, we want to find $q \leq p$ which forces $|f ``H_0 \times H_1| \geq 2$. For each $i < \mu^+$, we can choose $q_i \leq p$ which decides the value of the $i$-th element of $\dot{H}_{1}$ as $\xi_{i}$. There is a $K \in [\mu^{+}]^{\mu^{+}}$ such that $\forall i,j \in K(q_i \cdot q_j \not= 0)$. By the property of $f$, we can choose $\alpha,\beta \in H_0$ and $i,j\in K$ such that $f(\alpha,\xi_i) \not= f(\beta,\xi_j)$. Thus, $q_i\cdot q_j \force |f``H_0 \times \dot{H}_1| \geq 2$. 
\end{proof}

\begin{lem}\label{negationpreserved}
 If $\npolar{\mu^{++}}{\mu^{+}}{\mu^{+}}{\mu^{+}}{\mu}$ holds then any $(\mu^{+},\mu^{+},<\omega)$-c.c. poset forces the same partition relation.
\end{lem}
\begin{proof}
 Let $f$ be a coloring that witnesses with $\npolar{\mu^{++}}{\mu^{+}}{\mu^{+}}{\mu^{+}}{\mu}$. For each $p \force \dot{H}_{0} \in [\mu^{++}]^{\mu^{+}}$ and $\dot{H}_0 \in [\mu^{+}]^{\mu^{+}}$, we want to find $q \leq p$ which forces $|f ``H_0 \times \dot{H}_1| \geq 2$. For each $i < \mu^+$, we can choose $q_i \leq p$ which decides the value of the $i$-th element of $\dot{H}_{0}$ and $\dot{H}_1$ as $\xi_{i},\zeta_i$. There is a $K \in [\mu^{+}]^{\mu^{+}}$ such that $\forall i,i',j,j' \in K(q_i \cdot q_{i'} \cdot q_j \cdot q_{j'}\not= 0)$. By the property of $f$, we can choose $i,i',j,j' \in K$ such that $f(\xi_i,\zeta_{j}) \not= f(\xi_{i'},\zeta_{j'})$. Thus, $q_i \cdot q_{i'} \cdot q_j \cdot q_{j'} \force |f``\dot{H}_0 \times \dot{H}_1| \geq 2$. 
\end{proof}

\begin{lem}\label{ppmainlemma3}
 Suppose that $U$ is a normal ultrafilter over $\mu$. If $\npolar{\mu^{++}}{\mu^{+}}{\nu}{\mu^{+}}{\mu}$ for some regular $\nu < \mu$ then $\mathcal{P}_{U}$ forces the same partition relation.
\end{lem}
\begin{proof}
We divide two cases $\gamma = \omega$ and $\gamma > \omega$. First, we assume $\gamma > \omega$.

Let $f:\mu^{++} \times \mu^{+} \to \mu$ be a coloring that witnesses with $\npolar{\mu^{++}}{\mu^{+}}{\gamma}{\mu^{+}}{\mu}$. For each $\langle a,X\rangle \force \dot{H}_{0} \in [\mu^{++}]^{\gamma}$ and $\dot{H}_{1} \in [\mu^{+}]^{\mu^{+}}$, we want to find an extension of $\langle a,X\rangle$ that forces $|f ``\dot{H}_{0} \times \dot{H}_1|\geq 2$. This $\dot{H}_0$ can be shrinked to be in $V$ by 
\begin{clam}
 $\mathcal{P}_{U}$ forces that if $\mu > {\mathrm{ot}}(A) = \mathrm{cf}(\gamma) = \gamma> \omega$ then there is $B \in V$ such that ${\mathrm{ot}}(A) = \mathrm{ot}(B)$ and $B \subseteq A$ for all $A \subseteq \mathrm{ON}$. 
\end{clam}
\begin{proof}[Proof of Claim]
 Consider $\langle a,X \rangle \force \dot{A} \subseteq \mathrm{ON}$ and $\mathrm{ot}(\dot{A}) = \gamma$, $\gamma \in (\omega,\mu)\cap \mathrm{Reg}$. For every $i < \gamma$, let $\mathcal{A}_{i}$ be a maximal anti-chain below $\langle a,X\rangle$ such that $\forall \langle b,Y\rangle \in \mathcal{A}_{i} \exists \xi\in \mathrm{ON}(\langle b,Y\rangle \force$ the $i$-th element in $\dot{A}$ is $\xi$). By Lemma \ref{prikrycondi}, there is a $Z_i \subseteq X$ and $n_i < \omega$ such that $\mathcal{B}_{i} = \{\langle b,Y\rangle \in \mathcal{A}_{i} \mid |b| = n_i\}$ is maximal anti-chain below $\langle a,Z_i \rangle$. There are $I \in [\gamma]^{\gamma}$ and $n$ such that $n_i = n$ for all $i \in I$. If $n \leq |a|$, letting $B = \{\xi \mid \langle a,Z_i\rangle \force $the $i$-th element of $\dot{A}$ is $\xi\}$, it is easy to see that $\langle a, \bigcap_{i \in K}Z_i \rangle \force B \subseteq \dot{A}$ and $\mathrm{ot}(B) = \gamma$. If $n > |a|$, Let $x\in [\bigcap_{i\in K}Z_i \setminus (\max{a} + 1)]^{n-|a|}$. If we let $b = a \cup x$ then it is forced that $B = \{\xi \mid \langle b,Z_i\rangle \force $the $i$-th element of $\dot{A}$ is $\xi\}$ works as a witness by $\langle b,\bigcap_{i}Z_i\rangle\leq \langle a,\bigcap_{i \in K} Z_i  \rangle $, as desired.
\end{proof}

  By the claim, there are $q \leq \langle{a,X} \rangle$ and $H_0$ such that $q \force H_0 \subseteq \dot{H}_{0}$ and ${\mathrm{ot}}(H_0) = \gamma$. For each $i < \mu^{+}$, we can choose $\langle c_{i},Z_{i}\rangle \leq q$ which forces that the $i$-th value of $\dot{H}_{1}$ is $\xi_{i}$. Then there are $K \in [\mu^{+}]^{\mu^{+}}$ and $c$ such that $c_{i} = c$ for all $i \in K$. By the property of $f$, we can choose $\alpha<\beta$ in $H_0$ and $i<j$ in $K$ such that $f(\alpha,\xi_i) \not= f(\beta,\xi_j)$. Now $\langle {c,Z_i \cap Z_j} \rangle$ forces $f(\alpha,\xi_i), f(\beta,\xi_j) \in f``\dot{H}_0 \times \dot{H}_1$.

 Let us show in the case of $\gamma = \omega$. Let $f$ be a coloring that witnesses with $\npolar{\mu^{++}}{\mu^{+}}{\omega}{\mu^{+}}{\mu}$. Let $\langle a,X \rangle \force \dot{H}_{0} \in [\mu^{++}]^{\omega}$ and $\dot{H}_{1} \in [\mu^{+}]^{\mu^{+}}$. For each $i < \mu^{+}$, we can choose $\langle c_{i},Z_i \rangle \leq \langle a,X \rangle$ which forces that the $i$-th value of $\dot{H}_1$ is $\xi_i$. Again, there are $K \in [\mu^{+}]^{\mu^{+}}$ and $c$ such that $c_i = c$ for all $i \in K$. There is a $\langle{c,Z} \rangle \leq \langle a,X \rangle$ which forces that $|\{i \in K \mid \langle c,Z_i \rangle \in \dot{G}\}| = \mu^{+}$. We claim that $\langle c,Z \rangle$ forces $|f``\dot{H}_0 \times \dot{H}_1| \geq 2$.

 First, we assume there are $\langle b,Y\rangle \leq \langle c,Z\rangle$, $J \in [\omega]^{\omega}$, and $H = \{\alpha_{n} \mid n \in J\}$ such that $\langle b,Y\rangle \force H \subseteq \dot{H}_0$. Note that $\{i < \mu^{+} \mid b \setminus c \subseteq Z_i\}$ is unbounded since $\langle b,Y\rangle \force \{i < \mu^{+} \mid \langle c,Z_i \rangle \in \dot{G}\}$ is unbounded. By the property of $f$, there are $i,j \in K$ and $n,m \in J$ such that $f(\alpha_n,\xi_i) \not= f(\alpha_m,\xi_j)$ and $b \setminus c \subseteq Z_i \cap Z_j$. $\langle b,Z\cap Z_i \cap Z_j \rangle \force f(\alpha_n,\xi_i), f(\alpha_m,\xi_j) \in f``\dot{H}_0 \times \dot{H}_1$. 

 Next, we assume there is no such $\langle b,Y\rangle$. Towards showing a contradiction, suppose that there is an extension $\langle b,Y \rangle$ of $\langle c,Z \rangle$ which forces $f``\dot{H}_0 \times \dot{H}_1 = \{\eta\}$ for some $\eta$. Let $\dot{\alpha}_{n}$ be a $\mathcal{P}_{U}$-name of the $n$-th element of $\dot{H}_{0}$. By Lemma \ref{prikrycondi}, for each $n < \omega$, there are $Y_n$ and $l(n)$ such that $\{\langle b',Y' \rangle \leq \langle b,Y_n\rangle \mid |b'| = l(n)$ and $\langle b',Y' \rangle$ decides the value of $\dot{\alpha}_n\}$ contains maximal anti-chain $\mathcal{A}_n$. Note that, for each $x \in [Y_{n} \setminus (\max{b} + 1)]^{l(n)-n}$, there are $\alpha_{n}^x < \mu^{++}$ and $Y_{n}^x$ such that $\langle b \cup x,Y_{n}^{x} \rangle \force \dot{\alpha}_n = \alpha_{n}^{x}$ and $\mathcal{A}_{n} = \{\langle b \cup x,Y_{n}^{x} \rangle  \mid x \in [Y_{n}\setminus (\max{b} + 1)]^{n-l(n)}\}$.

Then the nice name defined by $\bigcup_{x \in [Y_{n}\setminus (\max{b} + 1)]^{n-l(n)}}\{\langle \alpha_n^{x},\langle b\cup x,Y_{n}^{x}\rangle\rangle \}$ denotes $\dot{\alpha}_n$ below $\langle b,Y_n\rangle$.

Again, we note that $\{i \in K \mid b \setminus c \subseteq Z_i\}$ is unbounded. Let $\theta$ be a sufficiently large regular. Let $M \prec \mathcal{H}_{\theta}$ be an elementary substructure with the following conditions: 
\begin{itemize}
 \item ${^{\omega}M} \cup \mu \subseteq M$ and $f,\{Z_{i},\xi_i\mid i \in K\}, \{\langle\alpha_{n}^{x} \mid x \in [Y_{n}\setminus (\max{b} + 1)]^{n-l(n)} \rangle \mid n < \omega\},U \in M$. 
 \item $M \cap \mu^{+} = \delta < \mu^{+}$ and $|M| < \mu^{+}$. 
\end{itemize}
Note that there is a $\delta^{*} \geq \delta$ such that $b \setminus c \subseteq Z_{\delta^{*}}$ and $\delta^{*} \in K$. 
Because there is no extension of $\langle c,Z\rangle$ which forces $[\dot{H}_0]^{\omega} \cap V \not= \emptyset$, there is $n$ such that $\{\alpha_n^{x} \mid x \in [Y_{n} \cap Z_{\delta^{*}}\setminus (\max{b} + 1)]^{n-l(n)} \}$ is of size $\mu$. Fix $\{x_k \mid k < \omega \}\subseteq [Y_n \cap Z_{\delta^{\ast}} \setminus (\max{b} + 1)]^{l(n)}$ with $\alpha_{n}^{x_{k}} \not= \alpha_{n}^{x_{l}}$ for $k \not= l$. Because $\langle b \cup x_k,Z_{\delta^{\ast}}\cap Y_{n}\rangle$ is a common extension of $\langle c,Z_{\delta^{\ast}}\rangle$ and $\langle b\cup x_k,Y_{n} \cap Y_{n}^{x_k} \rangle$, that forces $\langle \alpha_{n}^{x_k},\xi_{\delta^{\ast}}\rangle \in \dot{H}_0 \times \dot{H}_1$. In particular, $f(\alpha_{n}^{x_k},\xi_{\delta^{\ast}}) = \eta$. 

 Since ${^{\omega}M} \subseteq M$, $H'_0 = \{\alpha_n^{x_k} \mid k < \omega \}\in M$. In $M$, the following holds:
\begin{center}
 $\forall i < \mu^{+}\exists j > i(f``H'_0 \times \{\xi_{j}\} = \{\eta\})$. 
\end{center}
From this, $H'_1 = \{\xi_{i} \mid f``H \times \{\xi_i\} = \eta\}$ is unbounded in $\mu^{+}$. Thus, $f``H'_0 \times H'_1=\{\gamma\}$. This contradicts the choice of $f$. 
\end{proof}

Let us show Theorem \ref{maintheorem2}.
\begin{proof}[Proof of Theorem \ref{maintheorem2}]
 (1) and (2) follow by Lemma \ref{ppmainlemma1}. (3), (4), and (5) follow by Lemmas \ref{ppmainlemma2}, \ref{negationpreserved}, and \ref{ppmainlemma3}, respectively. 
\end{proof}

For section \ref{mainsection}, we recall a theorem of Hajnal--Juhasz. Let $\mathrm{Add}(\mu,\lambda)$ be the set of all partial functions from $\lambda$ to $\mu$ of size $<\mu$. 
\begin{thm}[Hajnal--Juhasz]\label{hajnaljuhasz}
 For any regular $\mu$, if $\mathrm{Add}(\mu,\mu^{+})$ has the $\mu^{+}$-c.c. then $\mathrm{Add}(\mu,\mu^{+})\force \npolar{\mu^{++}}{\mu^{+}}{\mu}{\mu^{+}}{2}$. 
\end{thm}
\begin{proof}
 The same proof as in \cite[Theorem 5.37]{MR2768692} works.
\end{proof}

\begin{coro}\label{ppcoro}
 Suppose that $\mu^{<\mu} = \mu$ and $\polar{\mu^{++}}{\mu^{+}}{\nu}{\mu^{+}}{\mu}$ holds for all $\nu < \mu$. Then $\mathrm{Add}(\mu,\mu^{+})$ forces that $\polar{\mu^{++}}{\mu^{+}}{\nu}{\mu^{+}}{\mu}$ holds for all $\nu < \mu$ but $\polar{\mu^{++}}{\mu^{+}}{\mu}{\mu^{+}}{2}$ fails. 
\end{coro}
\begin{proof}
 Note that if $\mu^{<\mu} = \mu$ then the product forcing $\prod_{\alpha < \mu^{+}}^{<\mu}P_{\alpha}$ of $(\mu,<\mu)$-centered posets $P_{\alpha}$ is $(\mu,<\mu)$-centered. For a proof, we refer to \cite[Lemma 4.2]{centeredkunen}. By $\mathrm{Add}(\mu,\mu^{+})\simeq \prod_{\alpha < \mu^{+}}^{<\mu}{2^{<\mu}}$, this is $(\mu,<\mu)$-centered. Lemmas \ref{ppmainlemma1} and \ref{hajnaljuhasz} show $\mathrm{Add}(\mu,\mu^{+})$ forces the desired partition relations.
\end{proof}
This answers \cite[Question 1.11]{MR4101445}. Note that this question has been solved in \cite{MR4094551} and \cite{gartishelah}. But our proof is the simplest of them. Indeed, 
\begin{coro}
 Suppose that $\lambda$ is an $\omega_1$-Erd\H{o}s cardinal. Then $\mathrm{Coll}(\omega_1,<\lambda) \times \mathrm{Add}(\omega,\omega_1)$ forces $\polar{\aleph_2}{\aleph_1}{n}{\aleph_1}{\aleph_0}$ for all $n < \omega$ and $\npolar{\aleph_2}{\aleph_1}{\aleph_0}{\aleph_1}{\aleph_0}$.
\end{coro}
\begin{proof}
  It is known that $\mathrm{Coll}(\omega_1,<\lambda)$ forces $(\omega_2,\omega_1) \twoheadrightarrow (\omega_1,\omega)$ if $\lambda$ is $\omega_1$-Erd\H{o}s. By Lemma \ref{ccimpliespp}, $\mathrm{Coll}(\omega_1,<\lambda)$ forces $\polar{\aleph_2}{\aleph_1}{n}{\aleph_1}{\aleph_0}$ for all $n < \omega$. By Corollary \ref{ppcoro}, $\mathrm{Coll}(\omega_1,<\lambda) \times \mathrm{Add}(\omega,\omega_1) \simeq \mathrm{Coll}(\omega_1,<\lambda) \ast \dot{\mathrm{Add}}(\omega,\omega_1)$ forces the desired conditions.
\end{proof}
Using an almost-huge cardinal, we can show that 
\begin{coro}
 Suppose that $\mu$ is a regular cardinal below an almost-huge cardinal. Then there is a $\mu$-directed closed poset which forces that $\polar{\mu^{++}}{\mu^{+}}{\nu}{\mu^{+}}{\mu}$ for all $\nu < \mu$ and $\npolar{\mu^{++}}{\mu^{+}}{\mu}{\mu^{+}}{\mu}$. 
\end{coro}
\begin{proof}
 By \cite[Theorem 1.2]{preprint}, there is a $\mu$-directed closed poset $P$ which forces that $\mu^{+}$ carries a $(\mu^{++},\mu^{++},<\mu)$-saturated ideal and $2^{\mu} = \mu^{+}$. By Lemma \ref{saturatedimplypolarized}, it is forced by $P$ that $\polar{\mu^{++}}{\mu^{+}}{\nu}{\mu^{+}}{\mu}$ for all $\nu < \mu$. By Corollary \ref{ppcoro}, $P \ast \dot{\mathrm{Add}}(\mu,\mu^{+})$ is a required poset.
\end{proof}

\begin{rema}
We give more observations for the preservation of polarized partition relations. As we saw in Theorem \ref{ccimpliespp}, $\polar{\aleph_{2}}{\aleph_1}{2}{\aleph_1}{\aleph_0}$ follows from Chang's conjecture $(\aleph_2,\aleph_1) \twoheadrightarrow (\aleph_1,\aleph_0)$. It is known that $(\aleph_2,\aleph_1) \twoheadrightarrow (\aleph_1,\aleph_0)$ is c.c.c. indestructible. On the other hand, as known the result of Jensen, $\square_{\omega_1}$ shows there is a c.c.c. poset which adds a Kurepa tree on $\omega_1$. Therefore, by Theorem \ref{kurepaimpliesnpp}, $\polar{\aleph_{2}}{\aleph_1}{2}{\aleph_1}{\aleph_0}$ can be destroyed by c.c.c. poset. Indeed, $\polar{\aleph_{2}}{\aleph_1}{2}{\aleph_1}{\aleph_0}$ is compatible with $\square_{\omega_1}$ (For example, \cite[Theorem 8.54]{MR2768692} and Lemma \ref{ccimpliespp} show). 
\end{rema}

\section{Proof of Theorem \ref{maintheorem3}}\label{mainsection}
In this section, we devoted to 
\begin{proof}[Proof of Theorem \ref{maintheorem3}]
 Starting a model with a supercompact cardinal $\mu$ below a huge cardinal $\kappa$. By Theorem \ref{laverind}, we may assume that $\mu$ is indestructible.

 By \cite{MR925267}, there is a $\mu$-directed closed poset $P$ which forces that $\mu^{+} = \kappa$ carries an ideal $\dot{J}$ such that
\begin{enumerate}
 \item $\dot{J}$ is centered. 
 \item $\dot{J}$ is $(\dot{\mu^{++}},\dot{\mu^{++}},<\mu)$-saturated. 
 \item $(\mu^{++},\mu^{+}) \twoheadrightarrow_{\mu} (\mu^{+},\mu)$. 
 \item $2^{\mu}= \mu^{+}$. 
\end{enumerate}
For (1) and (4), we refer to~\cite{MR925267}. (2) follows by~\cite[Section 3]{preprint}. (3) follows by the generic elementary embedding in~\cite{MR925267} and Lemma \ref{changsuff}.

Let $G \ast H$ be a $(V,P \ast \dot{\mathrm{Add}}(\mu,\mu^{+}))$-generic. Note that $\mathrm{Add}(\mu,\mu^{+})$ is $\mu$-directed closed and $\mu$-centered. In $V[G][H]$, the following holds:
\begin{enumerate}
 \item There is an ideal $I = \overline{\dot{J}^{G}}$ over $\mu^{+}$, which is centered.
 \item $I$ is $(\mu^{++},\nu,\nu)$-saturated, which in turn implies $\polar{\mu^{++}}{\mu^{+}}{\nu}{\mu^{+}}{\mu}$ for all $\nu < \mu$. 
 \item $\npolar{\mu^{++}}{\mu^{+}}{\mu}{\mu^{+}}{2}$.
 \item $(\mu^{++},\mu^{+}) \twoheadrightarrow_{\mu} (\mu^{+},\mu)$.
 \item $\mu$ is still supercompact and $2^{\mu} = \mu^{+}$. 
\end{enumerate}
(1) follows by Lemma \ref{termcentered}. (2) and (4) follow by Lemmas \ref{saturationinprikry} and \ref{ccpreserved}, respectively. (3) follows by Theorem \ref{hajnaljuhasz}. 
Let $U$ be a normal ultrafilter over $\mu$. By $2^{\mu} = \mu^{+}$ and Lemma \ref{guidinggeneric}, there is a guiding generic $\mathcal{G}$. $\mathcal{P}_{U,\mathcal{G}}$ forces that
\begin{enumerate}
 \item $\mu = \aleph_{\omega}$.
 \item $\overline{I}$ is an ideal over $\mu^{+} = \aleph_{\omega+1}$ that is centered but \emph{not} layered.
 \item $\overline{I}$ is $(\aleph_{\omega+2},\aleph_{n},\aleph_{n})$-saturated, which in turn implies $\polar{\aleph_{\omega+2}}{\aleph_{\omega+1}}{\aleph_{n}}{\aleph_{\omega+1}}{\aleph_{\omega}}$ for all $n < \omega$.
 \item $\npolar{\aleph_{\omega+2}}{\aleph_{\omega+1}}{\aleph_{\omega+1}}{\aleph_{\omega+1}}{\aleph_{\omega}}$. In particular, $\overline{I}$ is \emph{not} strongly saturated.
 \item $(\aleph_{\omega+2},\aleph_{\omega+1}) \twoheadrightarrow (\aleph_{\omega+1},\aleph_{\omega})$.

\end{enumerate}
(2) follows by Theorem \ref{maintheorem2}. (3), (4), and (5) follow by Lemmas \ref{saturationinprikry}, \ref{negationpreserved}, and \ref{ccpreserved}, respectively. 

Let $\dot{U}$ and $\dot{\mathcal{G}}$ be $P \ast \dot{\mathrm{Add}}(\mu,\mu^{+})$-names for $U$ and $\mathcal{G}$. $P \ast \dot{\mathrm{Add}}(\mu,\mu^{+})\ast \mathcal{P}_{\dot{U},\dot{\mathcal{G}}}$ is a required poset.
\end{proof}
\begin{rema}
 In $V[G][H]$, the ideal $I$ is $(\mu^{++},\mu^{++},<\mu)$-saturated. Let us discuss in $V[G]$. Let $\dot{I}$ be $\mathrm{Add}(\mu,\mu^{+})$-name for $I$ and $J = \dot{J}$. Since $\mathrm{Add}(\mu,\mu^{+}) \force \mathcal{P}(\mu^{+}) / \dot{I} \simeq \mathcal{P}(\mu^{+})/J \ast \dot{j}(\mathrm{Add}(\mu,\mu^{+})) \simeq \mathcal{P}(\mu^{+})/J \times \mathrm{Add}(\mu,\mu^{++})$. Let $e$ be a complete embedding that is given by Theorem \ref{duality}. By $e(p) = \langle 1,p\rangle \in \mathcal{P}(\mu^{+})/J \times \mathrm{Add}(\mu,\mu^{++})$, 
\begin{center}
$\mathrm{Add}(\mu,\mu^{+}) \force \mathcal{P}(\mu^{+})/J \times \mathrm{Add}(\mu,\mu^{++}) /\dot{G} \simeq \mathcal{P}(\mu^{+})/J \times \mathrm{Add}(\mu,\mu^{++})$. 
\end{center}It is easy to see that $\mathcal{P}(\mu^{+})/J \times \mathrm{Add}(\mu,\mu^{++})$ has the $(\mu^{++},\mu^{++},<\mu)$-c.c. in the extension by $\mathrm{Add}(\mu,\mu^{+})$, as desired. 

This shows that the assumption of strong saturation (that is the $(\mu^{++},\mu^{++},\mu)$-saturation) in Theorem \ref{stronglysatimplypp} cannot be improved to the $(\mu^{++},\mu^{++},<\mu)$-saturation. 
\end{rema}

\section{$\aleph_{\omega+2}$-Centered Ideal over $[\aleph_{\omega+3}]^{\aleph_{\omega+1}}$ and $\mathrm{Tr}_{\mathrm{Chr}}(\aleph_{\omega+3},\aleph_{\omega+1})$}
In this section, we study the relation between centered ideals on and $\mathrm{Tr}_{\mathrm{Chr}}(\lambda,\kappa)$. $\mathrm{Tr}_{\mathrm{Chr}}(\lambda,\kappa)$ is the statement that every graph of size and chromatic number $\lambda$ has a subgraph of size and chromatic number $\kappa$.

Shelah~\cite{MR1117029} proved that $V = L$ implies the existence of a graph $\mathcal{G}$ of size and chromatic number $\mu^{+}$ with every subgraph of size $\leq \mu$ has countable chromatic number for every cardinal $\mu$. That is, $\mathrm{Tr}_{\mathrm{Chr}}(\mu^{+},\lambda)$ fails for all $\lambda \in [\aleph_1,\mu^{+})$ in $L$. Foreman and Laver~\cite{MR925267} proved the consistency of $\mathrm{Tr}_{\mathrm{Chr}}(\lambda^{+},\mu^{+})$ for each regular $\mu < \lambda$. Therefore we are interested in $\mathrm{Tr}_{\mathrm{Chr}}(\lambda^{+},\mu^{+})$ for singular $\mu$. 

Here, we show the consistency of $\mathrm{Tr}_{\mathrm{Chr}}(\aleph_{\omega+3},\aleph_{\omega+1})$. First, we check some ideal assumption implies $\mathrm{Tr}_{\mathrm{Chr}}(\lambda^{+},\mu^{+})$ in Lemma \ref{idealandtrchr}. Then we generalize Theorem \ref{maintheorem2} to an ideal over $Z \subseteq \mathcal{P}(X)$ (see Lemma \ref{generalizedpreservation}) using Theorem \ref{duality}. Lastly, we construct a model with a required ideal.

\begin{lem}\label{idealandtrchr}
 Suppose that $[\lambda^{+}]^{\mu^{+}}$ carries a normal, fine, $\mu^{+}$-complete $\lambda$-centered ideal. Then $\mathrm{Tr}_{\mathrm{Chr}}(\lambda^{+},\mu^{+})$ holds.
\end{lem}
\begin{proof}
 Let $P = \mathcal{P}([\lambda^{+}]^{\mu^{+}})/I$ and $G$ be a $(V,P)$-generic filter. In $V[G]$, there is an elementary embedding $j:V \to M$ such that 
\begin{itemize}
 \item $\mathrm{crit}(j) = \mu^{+}$. 
 \item $j(\mu^{+}) = \lambda^{+}$ and $j(\mu^{++}) = \lambda^{++}$. 
 \item $j ``\lambda^{+} \in M$. 
\end{itemize}
Let $\mathcal{G} = \langle \lambda^{+},\mathrel{E} \rangle \in V$ be a graph of chromatic number $\lambda^{+}$. Since $j ``\lambda^{+} \in M$, $j(\mathcal{G})$ has a subgraph that is isomorphic with $\mathcal{G}$ in $M$.  

We claim that the chromatic number of $\mathcal{G}$ is $j(\mu^{+})$ in $V[G]$. Fix a $P$-name $\dot{c}$ for a coloring $\mathcal{G} \to \mu$ and $p \in P$. Let $F:P \to \lambda$ be a centering function of $P$. Define $d:\mathcal{G} \to \mu\times \lambda$ by $d(x) = \langle\xi,\alpha\rangle$ if and only if $\exists q \leq p(F(q) = \alpha \land q \force \dot{c}(x) = \xi))$. Since the chromatic number of $\mathcal{G}$ is $\lambda^{+}$, there are $x,y \in \mathcal{G}$ such that $x \mathrel{E} y$ and $d(x) = d(y)$. By $d(x) = d(y)$ and the definition of $d$, we have a $q \leq p$ which forces that $\dot{c}(x) = \dot{c}(y)$. Thus $P$ forces that the chromatic number of $\mathcal{G}$ is $\lambda^{+} = \dot{j}(\mu^{+})$. 

By the elementarity of $j$, there is a subgraph $\mathcal{G}$ of size and chromatic number of $\mu^{+}$. The proof is completed.
\end{proof}

To obtain a model in which $\mathrm{Tr}_{\mathrm{Chr}}(\aleph_{\omega+3},\aleph_{\omega+1})$ holds, it is enough to construct a model with a $\aleph_{\omega+2}$-centered ideal over $[\aleph_{\omega+3}]^{\aleph_{\omega+1}}$. The following lemma is an analogue of Lemma \ref{termcentered} for general quotient forcings. 
\begin{lem}\label{prikryquotientcentered}
Suppose that $P$ is $(\mu,<\nu)$-centered, $Q$ is $(\lambda,<\nu)$-centered, and $\dot{R}$ is a $Q$-name for a $(\mu,<\nu)$-centered poset. We also assume that the mapping $\tau:P \to Q \ast \dot{R}$, which has the form of $\tau(p) = \langle 1,f(p)\rangle$, is complete and there is a $\langle P_{\alpha},\dot{R}_{\alpha} \mid \alpha < \mu\rangle$ such that $P_{\alpha}$ is a filter, $P = \bigcup_{\alpha<\lambda}P_\alpha$, $\dot{R}_{\alpha}$ is a $Q$-names for a $<\nu$-complete filter, and $Q \force f ``P_{\alpha} \subseteq\dot{R}_{\alpha}$ and $\bigcup_{\alpha}\dot{R}_{\alpha} = \dot{R}$. If $\lambda^{\mu} = \lambda$ then the term forcing $T(P,Q \ast \dot{R} / \dot{G})$ is $(\lambda,<\nu)$-centered. In particular, if $P$ is $\nu$-Baire then $P \force Q \ast \dot{R} / \dot{G}$ is $(\lambda,<\nu)$-centered.
\end{lem}
\begin{proof}
 We may assume that $Q$ is a complete Boolean algebra. Let $F:Q \to \lambda$ be a centering function.
 
We want to define a centering function $l:T(P,Q \ast \dot{R} / \dot{G})\to \lambda$. For each $\dot{p} \in T(P,Q \ast \dot{R} / \dot{G})$, we have a maximal anti-chain $\mathcal{A}_{\dot{p}} \subseteq P$ such that every $p \in \mathcal{A}_{\dot{p}}$ forces $\dot{p} = \langle q,\dot{r} \rangle$ for some $\langle q,\dot{r} \rangle \in Q \ast \dot{R}$. Note that $p$ is a reduct of $\langle q,\dot{r} \rangle$. 

Define $l(p)$ by $\langle F(q\cdot ||\dot{r} \in \dot{R}_{\alpha}||) \mid p \in \mathcal{A}_{\dot{q}}, \alpha < \mu, p \force \dot{p} = \langle q,\dot{r}\rangle\rangle$. Note that the size of $\mu$-centered posets is at most $2^{\mu}$. By the assumption, the number of $\mathcal{A}_p$ is at most $\lambda$. This observation enables us to identify the range of $l$ by $\lambda$. Suppose $l(\dot{p}_0) = \cdots = l(\dot{p}_{i}) = \cdots$ ($i < \zeta < \nu$). Put $\mathcal{A} = \mathcal{A}_{\dot{p}}$. It is enough to show that each $p \in \mathcal{A}$ forces $\prod_i \dot{p}_i \in Q \ast \dot{R} / \dot{G}$. Fix $p \in \mathcal{A}$. Then, for each $i < \zeta$, there is a $\langle q_i,\dot{r}_i\rangle \in Q \ast \dot{R}$ such that $p$ forces $\dot{p}_i = \langle q_i,\dot{r}_i\rangle$.

For every $r \leq p$, since $p$ is a reduct of $\langle q_i,\dot{r}_i\rangle$, there is an $\alpha < \mu$ such that $q_i \cdot ||\dot{r}_i\cdot f(r)\in \dot{R}_{\alpha}||\not= 0$ for some (any) $i<\zeta$. By $Q \force f ``P_{\alpha}\subseteq \dot{R}_{\alpha}$, $r \in P_{\alpha}$. Note that $\prod_i q_i \cdot ||f(r) \cdot \prod_i\dot{r}_i \in \dot{R}_{\alpha}||=\prod_i q_i \cdot ||\dot{r}_i \in \dot{R}_{\alpha}|| \not= 0$. $p$ forces that
\begin{align*}
\tau(r)\cdot \textstyle \prod_i \dot{p}_i & =  \tau(r)\cdot \textstyle \prod_{i}\langle q_i ,\dot{r}_i\rangle \\ 
&\geq  \langle 1,f(r) \rangle \cdot \textstyle \prod_{i}\langle q_i \cdot ||\dot{r}_i \in \dot{R}_\alpha||,\dot{r}_i\rangle \\
&= \langle 1,f(r) \rangle \cdot \langle \textstyle \prod_{i}q_i \cdot ||\prod_i\dot{r}_i\in \dot{R}_\alpha||,\prod_i \dot{r}_i\rangle \\
&= \langle \textstyle \prod_{i}q_i \cdot ||\prod_i \dot{r}_i \cdot f(r) \in \dot{R}_{\alpha}||,\prod_i\dot{r}_i \cdot r\rangle\\ & \not=0.
\end{align*}
The translation of lines three to four follows by $||f(r) \in \dot{R}_{\alpha}|| = 1$ and $|| \prod_{i} \dot{r}_{i} \in \dot{R}_{\alpha}|| \cdot ||f(r) \in \dot{R}_{\alpha}|| = ||\prod_{i} \dot{r}_{i} \cdot f(r) \in \dot{R}_{\alpha}|| \leq ||\prod_{i} \dot{r}_{i} \cdot f(r) \not= 0||$.
Therefore $p$ is a reduct of $\textstyle \prod_i \langle q_{i},\dot{r}_i\rangle$. $p$ forces $\textstyle \prod_i \langle q_{i},\dot{r}_i\rangle =\prod_i \dot{p}_i \in Q \ast \dot{R} / \dot{G}$. In particular, $\textstyle \prod_i \dot{p}_i$ in the term forcing and it is a lower bound of $\dot{p}_i$'s. By Lemma \ref{laverbasiclemma}, if $P$ is $\nu$-Baire then $P \force Q \ast \dot{R} / \dot{G}$ is $(\lambda,<\nu)$-centered.
\end{proof}

\begin{lem}\label{generalizedpreservation}
 Suppose that $I$ is a normal, fine, $\mu^{+}$-complete $(\lambda,<\nu)$-centered ideal over $Z\subseteq \mathcal{P}(\lambda')$. Let $P$ be a $(\mu,<\nu)$-centered and $\nu$-Baire poset. If $\lambda^{\mu}= \lambda \leq \lambda^{'}$ then $P \force \overline{I}$ is a normal, fine, $\mu^{+}$-complete $(\lambda,<\nu)$-centered ideal over $Z$.
\end{lem}
\begin{proof}
 We may assume that $P$ is a Boolean algebra. Let $\langle P_\alpha \mid \alpha < \lambda \rangle$ be a centering family of $P$. We may assume that each $P_{\alpha}$ is a filter. Let $e$ be a complete embedding from $P$ to $\mathcal{P}(Z) / I \ast \dot{j}(P)$ that sends $p$ to $\langle 1,\dot{j}(p)\rangle$ by Theorem \ref{duality}. Theorem \ref{duality} also gives $P \force \mathcal{P}(Z) / \overline{I} \simeq \mathcal{B}(\mathcal{P}(Z) / I \ast \dot{j}(P) / \dot{G})$. By elementarity of $\dot{j}$ and $\mathrm{crit}(\dot{j}) \geq (\mu^{+})^{V}$,  $\mathcal{P}(Z) / I \force \dot{j}(P) = \bigcup_{\alpha < \mu}\dot{j}(P_{\alpha})$ and $\dot{j} ``P_{\alpha} \subseteq j(\dot{P}_{\alpha})$. 

Let $\dot{R}_{\alpha}$ be a $\mathcal{P}(Z) / I$-name for $\dot{j}(P_{\alpha})$. Since $\mathcal{P}(Z) / I$ is $\lambda$-centered and $\lambda^{\mu} =\lambda$, we can apply Lemma \ref{prikryquotientcentered} to $e:P \to \mathcal{P}(Z) / I \ast \dot{j}(P)$. Therefore $P \force \mathcal{P}(Z) / \overline{I} \simeq \mathcal{B}(\mathcal{P}(Z) / I \ast \dot{j}(P) / \dot{G})$ is $(\lambda,<\nu)$-centered.
\end{proof}

\begin{proof}[Proof of Theorem \ref{maintheorem4}]
 We may assume that $\mu$ is indestructible supercompact and $\mathrm{GCH}$ holds above $\mu$ by Theorem \ref{laverind}. 

By~\cite{shioya2}, there is a $\mu$-directed closed poset which forces that $[\mu^{+++}]^{\mu^{+}}$ carries a normal, fine, $\mu^{+}$-complete $\mu^{+}$-centered ideal and $2^{\mu} = \mu^{+}$. Let us discuss in the extension by this poset. Let $I$ be such an ideal. By Lemma \ref{guidinggeneric},  we can define $\mathcal{P}_{U,\mathcal{G}}$. By Lemma \ref{generalizedpreservation}, $\mathcal{P}_{U,\mathcal{G}}$ forces $\overline{I}$ is a normal, fine, $\aleph_{\omega+1}$-complete $\aleph_{\omega+2}$-centered ideal over $([\aleph_{\omega+3}]^{\aleph_{\omega+1}})^{V}$. By $\mathcal{P}_{U,\mathcal{G}}$ forces $([\aleph_{\omega+3}]^{\aleph_{\omega+1}})^{V} \subseteq [\aleph_{\omega+3}]^{\aleph_{\omega+1}}$, we can see $\overline{I}$ an ideal over $[\aleph_{\omega+3}]^{\aleph_{\omega+1}}$.

By Lemma \ref{idealandtrchr}, $\mathcal{P}_{U,\mathcal{G}} \force \mathrm{Tr}_{\mathrm{Chr}}(\aleph_{\omega+3},\aleph_{\omega+1})$, as desired.
\end{proof}

By using Magidor forcing, we obtain
 \begin{thm}\label{maintheorem5}
 Suppose that $\kappa$ is a huge cardinal with target $\theta$, $\mu < \kappa$ is a supercompact cardinal. For regular cardinals $\nu < \mu < \kappa < \lambda< \theta$, there is a poset which forces that 
\begin{enumerate}
 \item $[\kappa,\lambda] \cap \mathrm{Reg}$ and $[\omega,\nu] \cap \mathrm{Reg}$ are not changed,
 \item $\kappa = \mu^{+}$, $\lambda^{+} = \theta$, $\mathrm{cf}(\mu) = \nu$, 
 \item $[\theta]^{\kappa}$ carries a normal, fine, $\kappa$-complete $\lambda$-centered ideal, and
 \item $\mathrm{Tr}_{\mathrm{Chr}}(\theta,\kappa)$.
\end{enumerate}
\end{thm}
\begin{proof}
  This follows by \cite{shioya2} and Theorems \ref{laverind} and \ref{magidorforcing}.
\end{proof}
We proved the consistency of $\mathrm{Tr}_{\mathrm{Chr}}(\lambda^{+},\mu^{+})$ for singular $\mu$ and regular $\lambda > \mu^{+}$. We ask 
\begin{ques}
 Is $\mathrm{Tr}_{\mathrm{Chr}}(\aleph_{\omega+2},\aleph_{\omega+1})$ consistent?
\end{ques}

We conclude this paper with the following observation about layeredness. 

The ideal $I$ over $[\mu^{+++}]^{\mu^{+}}$ in a proof of Theorem \ref{maintheorem4} is $S$-layered for some stationary subset $S \subseteq E^{\mu^{+++}}_{\mu^{++}}$. We show that $I$ is not $T$-layered for all $T\subseteq E^{\mu^{+++}}_{<\mu^{++}}$ and $\overline{I}$ is forced to be not $S$-layered for all $S\subseteq E^{\mu^{+++}}_{\mu^{++}}$. The former follows by \cite[Claim 3]{shioya2} and in \cite[Section 4,5]{preprint}. The latter follows by the proof of Theorem \ref{maintheorem1}. Therefore $\overline{I}$ is not $S$-layered for all stationary $S \subseteq \mu^{+++}$ in the final model. On the other hand, $[\lambda]^{\mu^{+}}$ may carry an $S$-layered ideal for some stationary $S \subseteq \lambda$ and singular $\mu$ if $\lambda$ is a limit cardinal. Indeed,
\begin{prop}\label{hugeprop}
 Suppose that $\kappa$ is a huge cardinal and $\mu < \kappa$ is a supercompact cardinal. Then there is a poset that forces that $\lambda$ is a Mahlo cardinal, $[\lambda]^{\mu^{+}}$ carries a $S$-layered ideal for some stationary $S\subseteq \lambda \cap \mathrm{Reg}$, and $\mu$ is a singular cardinal.
\end{prop}
\begin{proof}
 We may assume that $\mu$ is indestructible supercompact by Theorem \ref{laverind}. Let $j:V \to M$ be a huge embedding with critical point $\kappa$. Then $\mathrm{Coll}(\mu,<\kappa) \force [j(\kappa)]^{\kappa}$ carries a normal, fine, and $\kappa$-complete ideal $I$ such that $\mathcal{P}([j(\kappa)]^{\kappa}) / I \simeq \mathrm{Coll}(\mu,<j(\kappa))$ (See \cite[Example 7.25]{MR2768692}). Let $\dot{U}$ be a $\mathrm{Coll}(\mu,<\kappa)$-name for a normal ultrafilter over $\mu$. By Theorem \ref{duality}, $\mathrm{Coll}(\mu,<\kappa) \ast \mathcal{P}_{\dot{U}} \force \mathcal{P}([j(\kappa)]^{\kappa})/\overline{I}\simeq \mathrm{Coll}(\mu,<j(\kappa)) \ast \mathcal{P}_{j(\dot{U})} / \dot{G} \ast \dot{H}$. We claim that it is forced that $\mathrm{Coll}(\mu,<\kappa) \ast \mathcal{P}_{\dot{U}} \force \mathrm{Coll}(\mu,<j(\kappa)) \ast \mathcal{P}_{j(\dot{U})} / \dot{G} \ast \dot{H}$ is $(\mathrm{Reg} \cap j(\kappa))^{V}$-layered. For $\mathrm{Coll}(\mu,<j(\kappa))$-name $\dot{X}$ for a subset of $\mu$, there is a maximal anti-chain $\mathcal{A}_{\dot{X}}$ such that every $q \in \mathcal{A}_{\dot{X}}$ decides $\dot{X}\in j(\dot{U})$. Let $\rho(\dot{X})$ be the least $\alpha < j(\kappa)$ such that $\mathcal{A}_{\dot{X}}\subseteq \mathrm{Coll}(\mu,<\alpha)$. For $\beta < j(\kappa)$, define $\rho(\beta) < j(\kappa)$ by $\sup \{\rho(\dot{X}) \mid \dot{X}$ is $\mathrm{Coll}(\mu,<\beta)$-name for a subset of $\mu\}\cup\{2^{\beta}\}$. Let $C$ be a club generated by $\rho$. For every $\alpha \in C \cap \mathrm{Reg}$, $\mathrm{Coll}(\mu,<\alpha) \force \dot{U}_{\alpha}:=\dot{j}(U) \cap V[\dot{G}_{\alpha}]$ is an ultrafilter. Here, $\dot{G}_{\alpha}$ is the canonical name for a generic filter of $\mathrm{Coll}(\mu,<\alpha)$. By Lemma \ref{prikryequiv}, 
\begin{center}
 $\mathrm{Coll}(\mu,<\kappa) \ast \mathcal{P}_{\dot{U}} \lessdot \mathrm{Coll}(\mu,<\alpha) \ast \mathcal{P}_{\dot{U}_{\alpha}}\lessdot \mathrm{Coll}(\mu,<j(\kappa)) \ast \mathcal{P}_{j(\dot{U})}$.
\end{center}
 Then $\mathrm{Coll}(\mu,<\alpha) \ast \mathcal{P}_{\dot{U}_{\alpha}} /\dot{G} \ast \dot{H}\lessdot \mathrm{Coll}(\mu,<j(\kappa)) \ast \mathcal{P}_{j(\dot{U})} /\dot{G} \ast \dot{H}$ holds in the extension by $\mathrm{Coll}(\mu,<\kappa) \ast \mathcal{P}_{\dot{U}}$. Let $\dot{P}_{\alpha}$ be a $\mathrm{Coll}(\mu,<\kappa) \ast \mathcal{P}_{\dot{U}}$-name for $\mathrm{Coll}(\mu,<f(\alpha)) \ast \mathcal{P}_{\dot{U}_{f(\alpha)}}$, here $f(\alpha) = \min (C \cap \mathrm{Reg})^{V}\setminus \alpha$. $\langle \dot{P}_{\alpha}\mid \alpha < j(\kappa) \rangle$ is forced to satisfy the condition of (2) in Lemma \ref{charlayered}. 
\end{proof}
Lemma \ref{modificationprikrycondi} brings an analogue of Lemma \ref{prikryequiv} for $\mathcal{P}_{U,\mathcal{G}}$. We can replace ``$\mu$ is a singular cardinal'' with $\mu = \aleph_{\omega}$ in the statement of Proposition \ref{hugeprop}.

   \bibliographystyle{plain}
   \bibliography{ref}
\end{document}